\documentclass[english]{article}
\usepackage[margin=2cm, top=1.95cm, bottom=1.95cm]{geometry}
\usepackage[hyphens]{url}
\usepackage[hidelinks]{hyperref}
\usepackage{amsthm,amsfonts,amssymb,amsmath}
\usepackage{graphicx}
\usepackage{cleveref}
\usepackage[square,sort,comma,numbers]{natbib}

\newtheorem{theorem}{Theorem}[section]
\newtheorem{proposition}[theorem]{Proposition}
\newtheorem{observation}[theorem]{Observation}
\newtheorem{prop}[theorem]{Proposition}
\Crefname{prop}{Proposition}{Propositions}
\newtheorem{lemma}[theorem]{Lemma}

\newtheorem{qn}[theorem]{Question}

\newtheorem{claim}[theorem]{Claim}
\newtheorem{corollary}[theorem]{Corollary}
\newtheorem{definition}[theorem]{Definition}

\theoremstyle{definition}

\numberwithin{equation}{section}
\numberwithin{figure}{section}
\allowdisplaybreaks

\usepackage{cleveref}
\usepackage{babel}
\usepackage[T1]{fontenc}
\usepackage[utf8]{inputenc}

\newcommand{\eps}{\varepsilon}
\newcommand{\su}{\subseteq}

\newcommand{\Z}{\mathbb{Z}}
\renewcommand{\S}{\mathcal{S}}

\renewcommand{\Pr}{\mathbb{P}}
\newcommand{\EE}{\mathbb{E}}
\DeclareMathOperator{\U}{U}
\newcommand{\A}{\mathcal{A}}
\newcommand{\E}{\mathcal{E}}
\newcommand{\C}{\mathcal{C}}

\newcommand{\tr}{\operatorname{Tr}}
\newcommand{\sch}{\operatorname{Sch}}

\title{Essentially tight bounds for rainbow cycles in proper edge-colourings
}

\author{Noga Alon\thanks{Department of Mathematics, Princeton
University, Princeton, USA. Research supported in part by NSF grant
DMS-2154082 and by USA-Israel BSF grant 2018267. Email: \href{mailto:nalon@math.princeton.edu} {\nolinkurl{nalon@math.princeton.edu}}.} \and
Matija Buci\'c\thanks{Department of Mathematics, Princeton University, Princeton, USA. Part of this work was completed while this author was in addition affiliated with Institute for Advanced Study. Research supported in part by NSF Awards CCF-1900460 and DMS-2349013. Email: \href{mailto:mb5225@princeton.edu} {\nolinkurl{mb5225@princeton.edu}}.}
\and
Lisa Sauermann\thanks{Institute for Applied Mathematics, University of Bonn, Bonn, Germany. Part of this work was completed while this author was affiliated with Massachusetts Institute of Technology. Research supported in part by NSF Award DMS-2100157 and a Sloan Research Fellowship. Email: 
\href{mailto:sauermann@iam.uni-bonn.de} {\nolinkurl{sauermann@iam.uni-bonn.de}}.}
\and
Dmitrii Zakharov\thanks{Department of Mathematics, Massachusetts Institute of Technology, Cambridge, USA. Research supported in part by a Jane Street Graduate Research Fellowship. 
Email: \href{mailto:zakhdm@mit.edu}{\nolinkurl{zakhdm@mit.edu}}.}
\and
Or Zamir\thanks{Department of Computer Science, Tel Aviv University, Tel Aviv, Israel. Email: \href{mailto:orzamir@tauex.tau.ac.il} {\nolinkurl{orzamir@tauex.tau.ac.il}}.}}
 \date{}

\usepackage{xcolor}

\begin{document}
\maketitle

\begin{abstract}
An edge-coloured graph is said to be \emph{rainbow} if no colour appears more than once. Extremal problems involving rainbow objects have been a focus of much research over the last decade as they capture the essence of a number of interesting problems in a variety of areas. A particularly intensively studied question due to Keevash, Mubayi, Sudakov and Verstra\"ete from 2007 asks for the maximum possible average degree of a properly edge-coloured graph on $n$ vertices without a rainbow cycle. Improving upon a series of earlier bounds, Tomon proved an upper bound of $(\log n)^{2+o(1)}$ for this question. Very recently, Janzer--Sudakov and Kim--Lee--Liu--Tran independently removed the $o(1)$ term in Tomon's bound, showing a bound of $O(\log^2 n)$. We prove an upper bound of $(\log n)^{1+o(1)}$ for this maximum possible average degree when there is no rainbow cycle. Our result is tight up to the $o(1)$ term, and so it essentially resolves this question.

In addition, we observe a connection between this problem and several questions in additive number theory, allowing us to extend existing results on these questions for abelian groups to the case of non-abelian groups. 
\renewcommand{\thefootnote}{\fnsymbol{footnote}} 
\renewcommand{\thefootnote}{\arabic{footnote}} 

\end{abstract} 

\section{Introduction}

Roughly speaking, a general extremal question asks for the maximum or minimum possible value of a parameter of a large structure which suffices to guarantee a certain type of substructure. Such questions capture a vast number of interesting problems appearing naturally in a variety of areas in mathematics and beyond. The area of extremal graph theory is concerned with such questions when the structures of interest are graphs, and graphs can often also be used to encode the properties of other types of mathematical objects relevant in the context of extremal question. It is very common to consider colours on the edges of a graph, for example, this is the setting in the whole area of (graph) Ramsey theory. Colours on the edges can also be very important in order to keep track of additional data when using graphs to encode other mathematical objects and structures. For example, the papers \cite{matousek} and \cite{X-ray} approached some questions in discrete geometry by considering edge-coloured graphs (given $n$ points in the plane, they considered graphs whose vertices correspond to the given points and whose edges correspond to segments between the given points of a certain length and direction, with edges of the same colour corresponding to parallel segments). Another example is the classical Ryser--Brualdi--Stein Conjecture about Latin squares, which can be phrased as an extremal problem about edge-coloured graphs, see \cite{ryser} for more details (and also more examples). In all of these examples, it is relevant to consider \emph{proper}\footnote{An edge-colouring of a graph is called proper if at every vertex all incident edges have distinct colours.} edge-colourings.

Another interesting notion in the context of edge-colourings is being rainbow, with some surprising applications ranging from classical decomposition conjectures such as Ringel's conjecture \cite{rainbow-ringel} and the classical Hall-Paige conjecture in group theory (see e.g.\ \cite{Alp-Alexey}) to problems in bioinformatics \cite{rainbow-bioinformatics}. An edge-coloured graph is said to be \emph{rainbow} if no colour appears on more than one edge. There has been a lot of research on
questions involving rainbow objects, with a paper of Keevash, Mubayi, Sudakov and Verstra\"ete \cite{KMSV} from 2007 being particularly influential. They raised rainbow analogues of a number of the perhaps most classical questions in extremal graph theory, and a particularly tantalising question they asked is the following: How many edges are needed in order to guarantee that in any proper edge-colouring of an $n$-vertex graph one can find a rainbow cycle? The non-rainbow version of this is the simple, yet often useful, fact that one needs $n$ edges in an $n$-vertex graph in order to guarantee a cycle (this is optimal as evidenced by any $n$-vertex tree). In other words, if the average degree of a graph is at least $2$, then the graph must contain a cycle. The rainbow version of this question, however, seems substantially more difficult and has resisted numerous attempts over the years. It is known, already due to \cite{KMSV}, that the answer for the rainbow version is larger, and in particular that for an $n$-vertex graph an average degree of at least $\Omega(\log n)$ is needed to guarantee a rainbow cycle. This remains the best-known lower bound (up to the implicit constant factor). 
We will discuss the best-known construction in \Cref{sec:additive} and show how it arises in a perhaps more natural way. The upper bound has been the subject of a series of improvements, starting with an initial upper bound of Keevash, Mubayi, Sudakov and Verstra\"ete \cite{KMSV}, who showed that any properly edge-coloured graph on $n$ vertices without a rainbow cycle has average degree at most $O(n^{1/3})$ (i.e.\ it has at most $O(n^{4/3})$ edges). The first improvement of this bound was due to Das, Lee and Sudakov \cite{das-rainbow}, showing an upper bound of the form $e^{(\log n)^{1/2+o(1)}}$ for the average degree. This was in turn improved by Janzer \cite{oliver-rainbow} to $O(\log^4 n)$ and subsequently to a bound of the form $(\log n)^{2+o(1)}$ by Tomon \cite{tomon-rainbow}. Very recently, the $o(1)$ term in the exponent was removed independently by Janzer and Sudakov \cite{JS-rainbow} and by Kim, Lee, Liu and Tran \cite{KLLT-rainbow}, showing the current state-of-the-art bound of $O(\log^2 n)$. Our main result is an essentially tight answer to the question of Keevash, Mubayi, Sudakov and Verstra\"ete, determining the average degree needed in a properly edge-coloured graph on $n$ vertices in order to guarantee a rainbow cycle up to lower order terms.

\begin{theorem}\label{thm-main}
There exists a constant $C>0$ such that every properly edge-coloured graph on $n\ge 3$ vertices with average degree at least $C\cdot  \log n \cdot \log \log n$ contains a rainbow cycle.
\end{theorem}

Our result is tight up to the $\log \log n$ factor. In particular, our result shows that every properly edge-coloured graph on $n$ vertices without a rainbow cycle must have average degree at most $(\log n)^{1+o(1)}$, which is tight up to the $o(1)$-term. Our result can also be rephrased as saying that for a certain absolute constant $C>0$, every properly edge-coloured graph on $n$ vertices with at least $C\cdot  n\cdot \log n \cdot \log \log n$ edges contains a rainbow cycle. The maximal possible number of edges in a properly edge-coloured graph on $n$ vertices without a rainbow cycle may be viewed as the \emph{rainbow Tur\'{a}n number} of the family of all cycles. Our result shows an upper bound of $O(n\cdot \log n \cdot \log \log n)$ for this rainbow Tur\'{a}n number, which is optimal up to the $\log\log n$ factor.

The previous work on this question can be split into two fairly different proof approaches. The previous state-of-the-art results \cite{JS-rainbow,KLLT-rainbow} both follow the ``homomorphism counting'' approach pioneered by Janzer in \cite{oliver-rainbow}, which has found plenty of other applications, see e.g.\ \cite{geometry-hom-counting,erdos-disproof-hom-counting,regular-subhypergraphs-hom-counting}. The other approach of ``passing to an expander'', first applied to this problem by Das, Lee and Sudakov \cite{das-rainbow}, has also found a number of applications and was significantly refined by Tomon in \cite{tomon-rainbow}, see \cite{erdos-gallai} and references therein for concrete examples. Our argument falls in the latter camp, and in the first step of our proof we pass to a subgraph which is a robust sublinear expander (see \Cref{sec:expanders} for a definition) of essentially the same average degree as our initial graph. We note that our expander subgraph does not have quantitatively stronger expansion properties compared to that of \cite{tomon-rainbow}, and our method of finding it is actually very similar to that of \cite{tomon-rainbow}. In fact, in general, finding an expander subgraph with substantially better quantitative expansion properties is impossible due to constructions obtained in \cite{asaf-constructions}. 
The crucial difference in our argument compared to \cite{tomon-rainbow} is that in our setup the expansion properties are much more robust, in a sense that was first introduced in \cite{haslegrave-robust} and subsequently used and further developed in \cite{erdos-gallai}. 
We introduce various new tools for working with such expanders, which may also be useful for other applications. 
The main part of our argument for finding a rainbow cycle in such an expander subgraph is a carefully designed random process that we analyse relying on the (robust) expansion properties and delicate multiple exposure arguments. In particular, a key point of our argument is a setup where we condition on the current state of the random process while keeping both the past and the future sufficiently random. At every step, we then show that (with high likelihood) we have significant expansion of a certain ``rainbow-reachable'' vertex set when going to the next step, or we had significant expansion when compared to the previous step. Interestingly, we cannot guarantee significant expansion of this ``rainbow-reachable'' vertex set for every step, but only going forwards \emph{or} going backwards from any given step. Then, when comparing every other step, we have significant expansion and can complete our argument. See Section~\ref{sec:overview} for a more detailed proof sketch.

The bound obtained in \Cref{thm-main} represents a hard limit for our approach in several ways. We will discuss this, as well as some circumstantial evidence that the $\log \log n$ factor in \Cref{thm-main} might actually be necessary, in \Cref{sec:conc-remarks}.

We also obtain a result bridging the gap between the answers in the rainbow and non-rainbow versions of the problem.
A cycle in a properly edge-coloured graph is said to be \textit{$r$-almost rainbow} if no colour appears on more than $r$ edges of the cycle. So, for example, an $1$-almost rainbow cycle is the same as a rainbow cycle, and any cycle of length $\ell$ is trivially $\ell$-almost rainbow. Somewhat surprisingly, the behaviour of the problem is already different from the rainbow case, if we allow colours to be used twice. From that point on, as we allow colours to be used more and more times, the answer gradually changes up to the point where we allow each colour to be used $\Theta(\log n)$ times. At this point already a linear number of edges suffices, like in the non-rainbow version of the problem. The precise behaviour is captured in the following result. Here, we denote by $\tr_{r}(n,\C)$ the $r$-almost rainbow Tur\'an number for the family $\C$ of all cycles, namely the maximal possible number of edges in a properly edge-coloured $n$-vertex graph without an $r$-almost rainbow cycle. With this notation, \Cref{thm-main} can be restated as $\tr_{1}(n,\C)\le O(n\log n\log\log n)$.

\begin{proposition}\label{prop:almost-rainbow}
    For any integers $n$ and $r$ with $2\le r < \frac{1}{10}\log n$, we have 
    \[\tr_{r}(n,\C)=\Theta\left(n \cdot \frac{(\log n)/r}{\log ((\log n)/r)}\right).\]
\end{proposition}

This proposition determines the $r$-almost rainbow Tur\'an number $\tr_{r}(n,\C)$ up to constant factors whenever $2\le r < \frac{1}{10}\log n$. Note that this automatically implies $\tr_{r}(n,\C)=\Theta(n)$ for all $r\ge \frac{1}{10}  \log n$. Indeed, for sufficiently large $n$ we have $\tr_{r}(n,\C)\le \tr_{\lfloor (\log n)/10\rfloor}(n,\C)\le \Theta(n)$ whenever $r\ge \frac{1}{10}  \log n$. In the opposite direction, for any $n$ and $r$, it is clear that $\tr_{r}(n,\C)\ge n-1$ by considering an $n$-vertex tree. Hence, \Cref{prop:almost-rainbow} actually determines the $r$-almost rainbow Tur\'an number $\tr_{r}(n,\C)$ up to constant factors for any integers $n$ and $r$ with $r\ge 2$.

We note that several competing notions of almost rainbow cycles have been considered over the years.
In particular, Janzer and Sudakov \cite{JS-rainbow} consider what happens if only a small proportion of edges on a cycle are allowed to use non-unique colours. We note that \Cref{prop:almost-rainbow} can not be improved in this direction, not even if we only insist on having at least one colour appearing only once, see the discussion at the end of Section \ref{sec:almost-rainbow}.

\subsection{Applications in additive number theory}\label{sec:intro-additive}

Given an abelian group $(G,+)$, a subset $S \subseteq G$ is said to be \emph{dissociated} if there is no solution to $\eps_1 g_1+\dots+\eps_m g_m=0$ with distinct $g_1,\dots,g_m\in G$ as well as $\eps_1,\dots,\eps_m \in \{-1,1\}$ and $m\ge 1$. One can view this as saying that there is no non-trivial ``linear'' relation between elements in $S$ with coefficients $\pm1$. With this in mind, dissociated sets play to some extent a similar role in arbitrary abelian groups as linearly independent sets play in vector spaces. Taking the analogy even further, a maximal dissociated subset of a set $A \subseteq G$ spans the set $A$, using only coefficients $0$ or $\pm 1$, see \cite{sanders-old}. In addition, all maximal dissociated subsets of any set $A \subseteq G
$ have almost the same size, as shown in \cite{consistency-of-dimension}. It is therefore not surprising that controlling the size of a maximal dissociated subset is relevant for various problems in additive number theory, see \cite{add2,add1,add3,add4,add5,add6,add7,sanders_sumset,add8,add9}. Dissociated sets also play an important role in Harmonic analysis, see e.g.\ \cite{konyagin,rudin}. 
The above-mentioned properties lead naturally to a notion of dimension of a finite subset $A\su G$ of an abelian group $G$. More specifically, in \cite{additive_definition} Schoen and Skhredov define the \emph{additive dimension} $\dim A$ of $A$ to be the size of the largest dissociated subset of $A$. 
The following theorem of Sanders \cite{sanders-old} 
(see also \cite[Theorem 1]{additive_definition}) shows that a subset $A$ with small doubling $A+A:=\{a_1+a_2 \mid a_1,a_2 \in A\}$ has small additive dimension and hence needs to be quite structured (namely, it can be ``well approximated'' by a ``subspace'' generated by its maximal dissociated subset).

\begin{theorem}
\label{thm:sanders}
    Let $G$ be an abelian group and let $A \subseteq G$ be a finite subset of size $|A|\ge 2$. If $|A+A|\le K|A|$ for some positive integer $K$, then we have $\dim A \le O(K \log |A|)$.
\end{theorem}

The study of the structure of sets with small doubling traces back to a celebrated result of Freiman \cite{freiman} from 1964 and has been the subject of extensive study ever since, see e.g.\ \cite{abelian-freiman, sanders-abelian} and references therein. There has been a lot of work extending results about sets with small doubling to non-abelian groups, see a foundational paper of Tao \cite{terry-non-abelian} fuelling the work in this direction, and see \cite{non-abelian-1,non-abelian-2,non-abelian-3,non-abelian-4,non-abelian-5,non-abelian-6} for some further examples. 
Given this, it is very natural to ask whether the result of Sanders stated in Theorem~\ref{thm:sanders} can be extended to non-abelian groups. As a start, the above definition of additive dimension extends essentially verbatim to the non-abelian case, although we repeat it here in order to switch to the multiplicative notation (to highlight the fact we are working with general, not necessarily abelian groups).

\begin{definition}
    Let $(G,\cdot)$ be a group with identity element $e$. A subset $S \subseteq G$ is said to be \emph{dissociated} if there is no solution to $g_1^{\eps_1}\dotsm g_m^{\eps_m}=e$  with distinct $g_1,\ldots, g_m \in S$ as well as $\eps_1,\dots,\eps_m \in \{-1,1\}$ and $m\ge 1$. For a finite subset $A \subseteq G$, the \emph{additive dimension} $\dim A$ is defined as the maximum cardinality of a dissociated subset of $A$.
\end{definition}

As an easy corollary of \Cref{thm-main}, we obtain a version of Sanders' result for \emph{all} groups (not just abelian groups) with slightly weaker bounds.

\begin{theorem}
\label{thm:additive-main}
    Let $(G,\cdot)$ be a group and let $A \subseteq G$  be a finite subset of size $|A|>2$. If $|A\cdot A|\le K|A|$ for some positive integer $K$, then we have $\dim A \le O(K \log |A| \log \log |A|)$. 
\end{theorem}

Here, the doubling $A\cdot A$ is defined by $A\cdot A:=\{a_1\cdot a_2 \mid a_1,a_2 \in A\}$. We do not know whether the $\log \log |A|$ term is necessary in the theorem above, but we do know the situation in the non-abelian case is at least slightly different compared to the abelian case. Namely, in the case of finite abelian groups, there is a stronger bound compared to \Cref{thm:sanders} if we ask for the dimension of the whole group, namely  $\dim G\le \lceil\log_2 |G|\rceil$ (see \cite{pm-davenport-2}). On the other hand, this bound does not hold in the non-abelian case. Indeed, we have $\dim \S_3^k\ge \frac{3}{\log_2 6} \cdot \log_2 |\S_3^k|$ (see \Cref{obs:lower-bound-product}), where $\S_3^k$ is the $k$-fold direct product of $\S_3$, the symmetric group on three elements. For finite non-abelian groups $G$, an upper bound for the additive dimension $\dim G$ can be obtained from \Cref{thm:additive-main}. More precisely, \Cref{thm:additive-main} implies $\dim G\le O(\log |G|\log\log |G|)$ for any finite group $G$ of size $|G|\ge 3$, which is only slightly weaker than the bound in the abelian case.

We note that the additive dimension $\dim G$ of the whole (finite) group $G$ agrees with a natural variant of the classical Davenport constant of a group, called the plus-minus weighted Davenport constant and denoted by $D_{\pm}(G)$. There has been a significant amount of work on plus-minus weighted Davenport constants, see \cite{pm-davenport-2, pm-weighted-davenport} or the survey \cite{pm-survey}, and see also \cite{arithmetic-weighted-davenport} for several arithmetic interpretations. Our \Cref{thm:additive-main} implies $D_{\pm}(G) =\dim G\le O(\log |G| \log \log |G|)$ for any finite group $G$ of size $|G|\ge 3$, almost matching the upper bound $D_{\pm}(G)\le\lceil \log_2 |G| \rceil$ in the abelian case.

While our manuscript was in its final stages we heard of some exciting progress by Conlon, Fox, Pham, and Yepremyan \cite{CFPY} towards a conjecture of Alon \cite{alon-ramsey-cayley} concerning the existence of good Ramsey Cayley graphs for any group. A key step in their argument is a lemma counting the number of subsets $A$ of a given group with the property that $|A \cdot A^{-1}| \le K|A|$. An asymmetric version of \Cref{thm:additive-main} (see \Cref{thm:additive-main-asym}) gives a bound on the dimension of such sets which in turn implies a bound on the number of such subsets, which is only slightly weaker than the lemma in \cite{CFPY}.

\vspace{0.2cm}
\textbf{Remark.} While our manuscript was under review, our main result \Cref{thm-main} has been used as a key ingredient in a remarkable paper by Alrabiah and Guruswami \cite{coding-theory-paper} showing almost-tight bounds for the dimension of binary linear codes that are three-query locally correctable against a fixed fraction of errors.
We refer to \cite{coding-theory-paper} for the relevant definitions and a detailed history of this well-known problem in coding theory. Alrabiah and Guruswami also used our result to obtain improved bounds for the case of $r$-query locally correctable codes for odd $r\ge 5$. 
They drew attention to the fact that the previously known bounds in the setting of \Cref{thm-main} would not have led to any non-trivial bounds in this latter coding theory application at all.

\vspace{0.2cm}
\textbf{Notation.} For a graph $G$, we denote the set of edges by $E(G)$ and the set of vertices by $V(G)$. For a vertex subset $U\su V(G)$, we denote the induced subgraph on $U$ by $G[U]$. A proper edge-colouring of $G$ with a colour palette $\mathcal{C}$ is an assignment $\gamma: E(G)\to \mathcal{C}$ such that for any two adjacent edges $e\neq e'$ we have $\gamma(e)\ne \gamma(e')$. For two adjacent vertices $v,v'\in V$, we slightly abuse notation to write $\gamma(v,v')$ for the colour of the edge $vv'$.

We denote the average degree of a graph $G$ by $d(G)$.  The neighbourhood of a vertex subset $U\su V(G)$ is the set of vertices in $V(G)\setminus U$ that are adjacent to a vertex in $U$. For a subset $F\su E(G)$ and a vertex $v\in V(G)$, we denote by $\deg_F(v)$ the number of edges in $F$ incident to $v$ (i.e.\ the degree of $v$ in the subgraph described by the edges in $F$). Furthermore, $N_{G-F}(U)$ denotes the neighbourhood of a vertex subset $U\su V(G)$ in the graph $G-F$ resulting from $G$ when deleting the edges in $F$.

All our logarithms are to the base $e$ unless otherwise specified. We make use of the standard asymptotic notation, namely, for functions $f:\mathbb{R}_{>0}\to \mathbb{R}_{>0}$ and $g:\mathbb{R}_{>0}\to \mathbb{R}_{>0}$, we write $f=O(g)$ or $g=\Omega(f)$ if there exists an absolute constant $C$ such that $f(x)\le Cg(x)$ for all $x\in \mathbb{R}_{>0}$. We write $f=\Theta(g)$ if we have both $f=O(g)$ and $g=O(f)$. We also write $f=o(g)$ if $f(x)/g(x) \to 0$ as $x \to \infty$.

\section{Proof Overview}\label{sec:overview}

As mentioned above, the first step of our proof of Theorem \ref{thm-main} is to pass to a well-connected subgraph of $G$ with almost the same average degree as $G$. To make this precise, we work with a notion of  ``robust sublinear expansion'', we discuss the rich history of this notion in \Cref{sec:expanders}. In such a robust sublinear expander graph with $n$ vertices, for every non-empty vertex subset $U$ of size at most $n^{1-\eps}$ (for any $0\le \eps\le 1$) the neighbourhood of $U$ has size at least $(\eps/3)\cdot |U|$. Even more so, this remains true even when deleting up to $(\eps/3)\cdot d(G)\cdot |U|$ edges, where $d(G)$ denotes the average degree of the graph (so the expansion property is in a sense quite robust). We prove that in any graph $G$, one can find a subgraph of similar average degree which is a robust sublinear expander in this sense (for more details, see Section \ref{sec:expanders}). In order to prove Theorem \ref{thm-main}, we can first choose such a subgraph of $G$ with the aim of finding a rainbow cycle in this subgraph. This allows us to reduce Theorem \ref{thm-main} to the case that our graph is a robust sublinear expander.

Our approach to finding a rainbow cycle uses the following simple idea. Suppose we split the colour palette $\mathcal{C}$ of the proper edge-colouring of $G$ into two disjoint subsets $\mathcal{C}_1$ and $\mathcal{C}_2$, say randomly (assigning each colour to $\mathcal{C}_1$ or $\mathcal{C}_2$ independently uniformly at random). If there exist two distinct vertices $x$ and $y$, such that $y$ can be reached from $x$ with a rainbow walk with colours in $\mathcal{C}_1$ and also with a rainbow walk with colours in $\mathcal{C}_2$, then these two rainbow walks use disjoint sets of colours. Hence we can combine them to obtain a rainbow closed walk and hence a rainbow cycle. Here, by a rainbow walk, we mean a walk on which all edges are coloured using distinct colours.

Our approach to finding such vertices $x$ and $y$ is rather naive: We fix an arbitrary vertex $x$ of the graph $G$, and look at the set $\U(\mathcal{C}_1)$ of all vertices $y$ that can be reached by rainbow walks starting at $x$ with colours in $\mathcal{C}_1$. We similarly look at the set $\U(\mathcal{C}_2)$ of all vertices $y$ that can be reached by rainbow walks from $x$ with colours in $\mathcal{C}_2$. If both of the sets $\U(\mathcal{C}_1)$ and $\U(\mathcal{C}_2)$ are large enough (more precisely, if they each have size larger than $\frac{n+1}2$), they must overlap in some vertex $y\neq x$. So our goal is to prove that with probability at least, say, $\frac23$ we have $|\!\U(\mathcal{C}_1)|>\frac{n+1}2$. Then we analogously obtain $|\!\U(\mathcal{C}_2)|>\frac{n+1}2$ with probability at least $\frac23$, and so with probability at least $\frac13$ the sets $\U(\mathcal{C}_1)$ and $\U(\mathcal{C}_2)$ have a common vertex $y\neq x$ (meaning that we can find a rainbow cycle).

Of course, the crucial difficulty is to show that we indeed have $|\!\U(\mathcal{C}_1)|>\frac{n+1}2$ with sufficiently large probability. In order to show this, for various subsets $\A\su \mathcal{C}_1$, we consider the set $\U(\A)$ of vertices that can be reached by rainbow walks from $x$ with colours in $\A$. More precisely, we consider a randomized chain of subsets $\C_1=\A_0\supseteq \A_1\supseteq \dots\supseteq \A_T$ of the colours. Then $\U(C_1)=\U(\A_0)\supseteq \U(\A_1)\supseteq \dots\supseteq \U(\A_T)$, and we analyse how the size of $\U(\A_j)$ grows as we vary $j$ from $T$ to $0$. We always have $|\!\U(\A_T)|\ge 1$ (as $x\in \U(\A_T)$), and our goal is proving that $|\!\U(\A_0)|>\frac{n+1}2$ holds with probability at least  $\frac 23$.

The basic intuition is that each set $\U(\A_{j-1})$ is obtained from $\U(\A_{j})$ by ``expanding'' via the edges with colours in $\A_{j-1}\setminus \A_{j}$.  Note that $\U(\A_{j})\su \U(\A_{j-1})$. Furthermore, any vertex outside $\U(\A_{j})$ is in $\U(\A_{j-1})$, if it is adjacent to a vertex in $\U(\A_{j})$ via an edge with a colour in $\A_{j-1}\setminus \A_{j}$. If there are many such vertices outside $\U(\A_{j})$, then $|\!\U(\A_{j-1})|$ is quite a bit larger than $|\!\U(\A_{j})|$.

Assuming that our graph $G$ is a ``robust sublinear expander'' in the sense discussed above, the neighbourhood of the set $\U(\A_{j})$ has size at least  $(\eps/3)\cdot |\!\U(\A_{j})|$, where $0\le \eps\le 1$ is such that $|\!\U(\A_{j})| \le n^{1-\eps}$. Furthermore, there are a lot of edges between $\U(\A_{j})$ and its neighbourhood. Let us expose the random colour set $\A_{j}$, which determines $\U(\A_{j})$, and let us look at the edges between $\U(\A_{j})$ and its neighbourhood. If a lot of these edges have colours in $\mathcal{C}\setminus \A_{j}$, then one can expect a reasonable fraction of these edges to have colours in $\A_{j-1}\setminus \A_{j}$ (the colour set $\A_{j-1}$ is still random, conditioned on our exposed outcome of $\A_{j}$). Using this, one can now expect a reasonable fraction of the at least $(\eps/3)\cdot |\!\U(\A_{j})|$ vertices in the neighbourhood of $\U(\A_{j})$ to have an edge into $\U(\A_{j})$ with colour in $\A_{j-1}\setminus \A_{j}$. These vertices are part of $\U(\A_{j-1})$ but not of $\U(\A_{j})$, and so $|\!\U(\A_{j-1})|$ is indeed quite a bit larger than $|\!\U(\A_{j})|$.

However, it can also happen that only a few of the edges between $\U(\A_{j})$ and its neighbourhood have colours in $\mathcal{C}\setminus \A_{j}$. This means that a lot of the edges between $\U(\A_{j})$ and its neighbourhood have colours in $\A_{j}$. If $vv'$ is an edge with a colour $\gamma(v,v')\in\A_{j}$, and we have $v\in \U(\A_{j})$ but $v'\not\in \U(\A_{j})$, then the colour $\gamma(v,v')$ of the edge $vv'$ must appear on every rainbow walk from $x$ to $v$ with colours in $\A_{j}$ (indeed, otherwise we could append the edge $vv'$ to a rainbow walk from $x$ to $v$ with colours in $\A_{j}\setminus \{\gamma(v,v')\}$, and we would obtain a rainbow walk from $x$ to $v'$ with colours in $\A_{j}$). Now, if $\gamma(v,v')\not\in \A_{j+1}$, then $v$ cannot be reached by a rainbow walk from $x$ with colours in $\A_{j+1}\su \A_{j}\setminus \{\gamma(v,v')\}$, and so $v\not\in \U(\A_{j+1})$. If there are many edges between $\U(\A_{j})$ and its neighbourhood with colours in $\A_{j}$, then we expect a reasonable fraction of these edges to have colours in $\A_{j}\setminus \A_{j+1}$. Using this, one can expect to have a reasonable number of vertices $v\in \U(\A_{j})$ to have an edge into the neighbourhood of $\U(\A_{j})$ with a colour in $\A_{j}\setminus \A_{j+1}$. All of these vertices are part of $\U(\A_{j})$ but not of $\U(\A_{j+1})$. Hence $|\!\U(\A_{j+1})|$ is quite a bit smaller than $|\!\U(\A_{j})|$.

In both of these cases (both the case where a lot of the edges between $\U(\A_{j})$ and its neighbourhood have colours in $\mathcal{C}\setminus \A_{j}$ and the case where a lot of these edges have colours in $\A_{j}$), we can expect $|\!\U(\A_{j-1})|$ to be quite a bit larger than $|\!\U(\A_{j+1})|$ (note that we always have $|\!\U(\A_{j-1})|\ge |\!\U(\A_{j})|\ge |\!\U(\A_{j+1})|$). More precisely, we heuristically expect an inequality like $|\!\U(\A_{j-1})|\ge (1+\eps/400)\cdot |\!\U(\A_{j+1})|$ if $|\!\U(\A_{j})| \le n^{1-\eps}$. Iterating this, and using that we always have $|\!\U(\A_{T})|\ge 1$, we can expect $|\!\U(\A_{0})|=|\!\U(\mathcal{C}_1)|$ to be quite large, implying that $|\!\U(\mathcal{C}_1)|>\frac{n+1}2$ with sufficiently large probability.

However, there are various challenges in actually implementing this strategy. First, we have to be careful about the notion of having ``a lot of'' edges between $\U(\A_{j})$ and its neighbourhood with colours in $\mathcal{C}\setminus \A_{j}$, or ``a lot of''  such edges with colours in $\A_{j}$. Namely, the relevant condition here is not just about the total number of such edges, but about having many such edges which are not too clustered. This is because we want to conclude that we can expect many vertices in $\U(\A_{j-1})\setminus \U(\A_{j})$ or $\U(\A_{j})\setminus \U(\A_{j+1})$, respectively (and so it is important that our edges under consideration are not clustered at a small number of vertices). This technical issue is resolved by Lemma \ref{lemma-auxiliary-graph} below, which shows that there are many ``non-clustered'' edges between $\U(\A_{j})$ and its neighbourhood with colours in $\mathcal{C}\setminus \A_{j}$ or many such ``non-clustered'' edges with colours in $\A_{j}$.

A more severe issue is that we can only show $|\!\U(\A_{j-1})|$  to be quite a bit larger than $|\!\U(\A_{j+1})|$ in expectation. Indeed, there can be a lot of dependencies for different vertices for whether they are contained in the sets $\U(\A_{j-1})$ and $\U(\A_{j+1})$.  Thus, we cannot establish concentration for $|\!\U(\A_{j-1})|$ or $|\!\U(\A_{j+1})|$ and so we only obtain our desired conclusions for the expectations of these random variables. A priori, this is not a problem (it essentially suffices to show that the expectation of $|\!\U(\mathcal{C}_1)|$ is large enough). However, the issue is that we need the assumption $|\!\U(\A_{j})| \le n^{1-\eps}$ for our expansion properties, and it does not suffice to have this assumption in expectation (indeed, even if the expectation of $|\!\U(\A_{j})|$ is at most $n^{1-\eps}$, we could still have $|\!\U(\A_{j})| > n^{1-\eps}$ with large probability). Of course, if the expectation of $|\!\U(\A_{j})|$ is small, then by Markov's inequality $|\!\U(\A_{j})|$ is small with high probability. However, one loses so much in the quantitative bounds for ``small'' here that one cannot apply this argument at every step.

The solution to this issue is to only apply Markov's inequality for certain spaced-out indices. For those indices, we can then conclude from a bound for the expectation of $|\!\U(\A_{j})|$ that $|\!\U(\A_{j})|$ is small with sufficiently high probability (namely, $|\!\U(\A_{j})|\le n^{1-\eps}$ for some suitably chosen $\eps$ depending on $j$). Whenever this is the case, we also have $|\!\U(\A_{j'})|\le |\!\U(\A_{j})|\le n^{1-\eps}$ for all $j'\ge j$ and so we get the desired expansion properties for all such sets $\U(\A_{j'})$. However, this requires performing the arguments sketched above while conditioning on $|\!\U(\A_{j})|\le n^{1-\eps}$. This conditioning makes the argument more complicated and introduces various technical difficulties. Luckily, these difficulties can be overcome and we can carry out the above strategy even with this additional conditioning.

\section{Robust sublinear expanders}\label{sec:expanders}

In our proof of \Cref{thm-main}, we will first reduce to the case that $G$ is a ``robust sublinear expander'' in the following sense. 

\begin{definition}\label{definition-robust-sublinear-expander}
A graph $G$ on $n\ge 1$ vertices is called a \emph{robust sublinear expander} if for every $0\le \eps\le 1$ and every non-empty subset $U\su V(G)$ of size  $|U|\le n^{1-\eps}$ the following holds. For every subset $F\su E(G)$ of $|F|\le (\eps/3)\cdot d(G)\cdot |U|$ edges, we have $|N_{G-F}(U)|\ge (\eps/3)\cdot  |U|$.
\end{definition}

Recall that $d(G)$ denotes the average degree of $G$ and $N_{G-F}(U)$ is the neighbourhood of $U$ in the graph resulting from $G$ when deleting the edges in $F$. Note that every robust sublinear expander must have at least two vertices (since for $n=1$ the condition in the definition above is not satisfied for $\eps=1$ and $U= V(G)$).

Various competing definitions of robust sublinear expander graphs have been considered over the past 30 years, with numerous applications \cite{shapira2015small,montgomery2015logarithmically,fernandez2022nested,fernandez2022build,chakraborti2021well,haslegrave2021crux,haslegrave2021ramsey,kim2017komlos,liu2017mader,liu2020solution,liu2020clique,letzter2023immersion,shobro,erdos-gallai,haslegrave-robust,tomon-rainbow,sudakov2022extremal,jiang2021rainbow}. This general idea as well as one of the first definitions are due to Koml\'os and Szemer\'edi \cite{K-S}. The above definition is a simplified version of the one in \cite{haslegrave-robust} in a similar manner as used in \cite{erdos-gallai}. The following lemma is a strengthening of an observation dating back to the original paper of Koml\'os and Szemer\'edi. Our proof for this lemma is motivated by \cite{tomon-rainbow}, where a similar lemma is obtained but for a weaker notion of sublinear expanders (while \cite{tomon-rainbow} considered a different type of expression to maximize, essentially the same maximization as in our proof below also occurs in the paper \cite{wang} optimizing the arguments from \cite{tomon-rainbow}).

The lemma states that in any graph $G$ we can find a subgraph $H$ with similar average degree as $G$, such that $H$ is a robust sublinear expander. More precisely, with ``similar average degree'' we mean that the ratio $d(H)/\log |V(H)|$ is almost the same as the ratio $d(G)/\log |V(G)|$ (up to constant factors).

\begin{lemma}\label{lemma-find-robust-expander-subgraph}
For every graph $G$ with at least one edge, there exists a  subgraph $H$ which is a robust sublinear expander and has average degree
\[d(H)\ge
\frac{\log |V(H)|-(1/3)}{\log |V(G)|-(1/3)}\cdot d(G)\ge 
\frac{1}{3}\cdot \frac {\log |V(H)|}{\log |V(G)|}\cdot d(G).\]
\end{lemma}

When applying this lemma, we will only use the lower bound for $d(H)$ given by the term on the right-hand side, but we include the stronger middle term in the statement in case it becomes helpful for future applications.

\begin{proof}[Proof of \Cref{lemma-find-robust-expander-subgraph}]
Among all subgraphs $H$ of $G$ (with at least one vertex), let us choose a subgraph $H$ maximizing the expression
\[\frac{d(H)}{\log |V(H)|-(1/3)}.\]
Then we clearly have
\[\frac{d(H)}{\log |V(H)|-(1/3)}\ge \frac{d(G)}{\log |V(G)|-(1/3)}.\]
In particular, as $d(G)>0$ and $\log |V(G)|\ge \log 2>1/3$, we must have $d(H)>0$ and $|V(H)|\ge 2$. So we indeed obtain
\[d(H)\ge \frac{\log |V(H)|-(1/3)}{\log |V(G)|-(1/3)}\cdot d(G)\ge \frac{1}{3}\cdot \frac {\log |V(H)|}{\log |V(G)|}\cdot d(G),\]
where the second inequality follows from $\log |V(H)|-(1/3)\ge (1/3)\cdot \log |V(H)|$ since $\log |V(H)|\ge \log 2 \ge 1/2$.

It remains to show that $H$ is a robust sublinear expander. Let $m=|V(H)|\ge 2$. Consider $0\le \eps\le 1$, and let $U\su V(H)$ be a non-empty subset of size $|U|\le m^{1-\eps}$. Furthermore, let $F\su E(H)$ be a subset of $|F|\le (\eps/3)\cdot d(H)\cdot |U|$ edges. We need to show that $|N_{H-F}(U)|\ge (\eps/3)\cdot |U|$.

Suppose for contradiction that $|N_{H-F}(U)|< (\eps/3)\cdot |U|$. Note that the number of edges in the graph $H$ equals $\frac{1}{2}\cdot  d(H)\cdot m$.

First, we claim that the subgraph of $H$ induced by the vertex set $\overline{U}=V(H)\setminus U$ contains at most $\frac{1}{2}\cdot d(H)\cdot |\overline{U}|$ edges. Indeed, if $|U|=m$ or $|U|=m-1$, this induced subgraph has at most one vertex and therefore it does not contain any edges. If $|U|\le m-2$, then $0<\log |\overline{U}|-(1/3)\le \log m -(1/3)$. By our choice of $H$, we have
\[\frac{d(H[\overline{U}])}{\log |\overline{U}|-(1/3)}\le \frac{d(H)}{\log m-(1/3)},\]
and therefore $d(H[\overline{U}])\le d(H)$. Now, the number of edges of the subgraph $H[\overline{U}]$ induced by $\overline{U}$ is $\frac{1}{2}\cdot d(H[\overline{U}])\cdot |\overline{U}|\le \frac{1}{2}\cdot d(H)\cdot |\overline{U}|$, as claimed.

Since the graph $H$ has at most $\frac{1}{2}\cdot d(H)\cdot |\overline{U}|=\frac{1}{2}\cdot d(H)\cdot (m-|U|)$ edges inside the set $\overline{U}=V(H)\setminus U$, there must be at least $\frac{1}{2}\cdot d(H)\cdot |U|$ edges in the graph $H$ that are incident with at least one vertex in $U$. Hence, there must be at least 
\[\frac{1}{2}\cdot d(H)\cdot |U|-|F|\ge \frac{1}{2}\cdot d(H)\cdot |U|-\frac{\eps}{3}\cdot d(H)\cdot |U|=\left(\frac{1}{2}-\frac{\eps}{3}\right)\cdot d(H)\cdot |U|\]
edges in the graph $H-F$ that are incident with at least one vertex in $U$. By definition of $N_{H-F}(U)$, all of these edges are contained in the set $X:=U\cup N_{H-F}(U)$. Hence, the induced subgraph $H[X]$ on $X$ has average degree 
\[d(H[X])=\frac{2\cdot |E(H[X])|}{|X|}\ge \frac{2}{|X|}\cdot \left(\frac{1}{2}-\frac{\eps}{3}\right)\cdot d(H)\cdot |U|=\left(1-\frac{2\eps}{3}\right)\cdot d(H)\cdot \frac{|U|}{|X|}.\]
By our assumption $|N_{H-F}(U)|< (\eps/3)\cdot |U|$, we have $|X|=|U|+|N_{H-F}(U)|< (1+(\eps/3))\cdot |U|$. So we can conclude
\[d(H[X])> \left(1-\frac{2\eps}{3}\right)\cdot d(H)\cdot \frac{|U|}{(1+(\eps/3))\cdot |U|}=\frac{1-(2\eps/3)}{1+(\eps/3)}\cdot d(H)\ge (1-\eps)\cdot d(H).\]
In particular, we obtain $d(H[X])>0$, so $X$ contains at least one edge and in particular $|X|\ge 2$. Together with
\[\log |X|<\log \left(1+\frac{\eps}{3}\right)+\log |U|\le \frac{\eps}{3}+\log(m^{1-\eps})=\frac{\eps}{3}+(1-\eps)\cdot \log m,\]
we now obtain
\[\frac{d(H[X])}{\log |X|-(1/3)}>\frac{(1-\eps)\cdot d(H)}{(\eps/3)+(1-\eps)\cdot \log m-(1/3)}=\frac{(1-\eps)\cdot d(H)}{(1-\eps)\cdot \log m-(1-\eps)\cdot (1/3)}=\frac{d(H)}{\log |V(H)|-(1/3)}.\]
This is a contradiction to $H$ maximizing the expression on the right-hand side, which finishes the proof.
\end{proof}

\section{Preliminaries}

\subsection{A lemma for robust sublinear expanders}

We will use the assumption that our graph is a robust sublinear expander by applying the following lemma (the property in the lemma is actually equivalent to the property of $G$ being a robust sublinear expander up to changing the constants appearing in Definition \ref{definition-robust-sublinear-expander}).

\begin{lemma}\label{lemma-auxiliary-graph}
Let $G$ be a robust sublinear expander (in the sense of Definition \ref{definition-robust-sublinear-expander}) with $n$ vertices, let $0\le \eps\le 1$ and let  $U\su V(G)$ be a vertex subset of size  $|U|\le n^{1-\eps}$. Consider a (not necessarily proper) colouring of the edges of $G$ between $U$ and $V(G)\setminus U$ with the colours red and blue. Then at least one of the following two statements holds:
\begin{itemize}
\item we can find a subset $F_{\text{red}}\su E(G)$ of red edges in $G$ with $|F_{\text{red}}|\ge (\eps/7)\cdot d(G)\cdot |U|$ and $\deg_{F_{\text{red}}}(v)\le \lceil d(G)\rceil$ for all $v\in U$, or
\item we can find a subset $F_{\text{blue}}\su E(G)$ of blue edges in $G$ with $|F_{\text{blue}}|\ge (\eps/7)\cdot d(G)\cdot |U|$ and $\deg_{F_{\text{blue}}}(v')\le \lceil d(G)\rceil$ for all $v'\in V(G)\setminus U$.
\end{itemize}
\end{lemma}

\begin{proof}
Let $E_{\text{red}}\su E(G)$ be the set of red edges in $G$, and $E_{\text{blue}}\su E(G)$ be the set of blue edges in $G$. Recall that all edges in $E_{\text{red}}$ and all edges in $E_{\text{blue}}$ are between $U$ and $V(G)\setminus U$, so they are of the form $(v,v')$ with $v\in U$ and $v'\in V(G)\setminus U$.

Let $F_{\text{red}}'\su E_{\text{red}}$ be the set of red edges $(v,v')\in E_{\text{red}}$ with $\deg_{E_{\text{red}}}(v)\le d(G)$. Similarly, let  $F_{\text{blue}}'\su E_{\text{blue}}$ be the set of blue edges $(v,v')\in E_{\text{blue}}$ with $\deg_{E_{\text{blue}}}(v')\le d(G)$. If $|F_{\text{red}}'|\ge (\eps/7)\cdot d(G)\cdot |U|$, then $F_{\text{red}}'$ satisfies the first alternative in the lemma statement. Similarly, if $|F_{\text{blue}}'|\ge (\eps/7)\cdot d(G)\cdot |U|$, then $F_{\text{blue}}'$ satisfies the second alternative in the lemma statement. Hence, we may assume that $|F_{\text{red}}'|< (\eps/7)\cdot d(G)\cdot |U|$ and $|F_{\text{blue}}'|< (\eps/7)\cdot d(G)\cdot |U|$, and therefore $|F_{\text{red}}'\cup F_{\text{blue}}'|< (2\eps/7)\cdot d(G)\cdot |U|$.

Let $U_{\text{red}}\su U$ be the set of vertices $v\in U$ with $\deg_{E_{\text{red}}}(v)\ge d(G)$, and let $V_{\text{blue}}\su V(G)\setminus U$ be the set of vertices $v'\in V(G)\setminus U$ with $\deg_{E_{\text{blue}}}(v')\ge d(G)$. Note that every edge $(v,v')\in E(G)$ with $v\in U\setminus U_{\text{red}}$ and $v'\in V(G)\setminus(U\cup  V_{\text{blue}})$ is contained in $F_{\text{red}}'$ or $F_{\text{blue}}'$ (depending on whether the edge is red or blue).

If $|U_{\text{red}}|\ge (\eps/7)\cdot |U|$, then we can choose a set $F_{\text{red}}$ as in the first alternative in the lemma statement by choosing   $\lceil d(G)\rceil$ edges $(v,v')\in E_{\text{red}}$ for each $v\in U_{\text{red}}$. Similarly, if  $|V_{\text{blue}}|\ge (\eps/7)\cdot |U|$, then we can choose a set $F_{\text{blue}}$ as in the second alternative by choosing   $\lceil d(G)\rceil$ edges $(v,v')\in E_{\text{red}}$ for each $v'\in V_{\text{blue}}$. Hence, we may assume that $|U_{\text{red}}|< (\eps/7) \cdot |U|$ and $|V_{\text{blue}}|< (\eps/7)\cdot |U|$, and therefore in particular $|U_{\text{red}}\cup V_{\text{blue}}|< (2\eps/7)\cdot |U|$.

We will now show that not all of these assumptions can hold at the same time by deriving a contradiction to $G$ being a robust sublinear expander. To obtain this contradiction, let $U'=U\setminus U_{\text{red}}$, then $|U'|\ge (1-\eps/7)\cdot |U|\ge (6/7)\cdot |U|$. Recall that all edges between $U'=U\setminus U_{\text{red}}$ and $V(G)\setminus(U\cup  V_{\text{blue}})$ are contained in $F_{\text{red}}'$ or $F_{\text{blue}}'$. Hence, setting $F=F_{\text{red}}'\cup F_{\text{blue}}'$, we have $N_{G-F}(U')\su U_{\text{red}}\cup V_{\text{blue}}$ and hence $|N_{G-F}(U')|\le |U_{\text{red}}\cup V_{\text{blue}}|< (2\eps/7)\cdot |U|\le (\eps/3)\cdot |U'|$. On the other hand, $|U'|\le |U|\le n^{1-\eps}$ and $|F|=|F_{\text{red}}'\cup F_{\text{blue}}'|<(2\eps/7)\cdot d(G)\cdot |U|\le (\eps/3)\cdot d(G)\cdot |U'|$, so this is a contradiction to $G$ being a robust sublinear expander.
\end{proof}

\subsection{Probabilistic lemmas}

In this subsection we collect some simple probabilistic lemmas which will be used in the proof of Theorem~\ref{thm-main}. The first lemma gives bounds for the conditional probabilities for various colours to be contained in various sets in the randomized nested colour set sequence $\mathcal{C}\supseteq \A_0\supseteq \A_1\supseteq \dots\supseteq \A_T$ mentioned in the proof overview in Section \ref{sec:overview} (when conditioning on the outcome of some $\A_j$).

\begin{lemma}\label{lemma-probability-nested-sequence}
For a given set $\mathcal{C}$, some integer $T$ and $1-(1/T)\le p<1$, consider a random sequence of nested subsets $\mathcal{C}\supseteq \A_0\supseteq \A_1\supseteq \dots\supseteq \A_T$ obtained as follows: Let $\A_0\su \mathcal{C}$ be a random subset obtained by including each element of $\mathcal{C}$ independently with probability $1/2$. For any outcome of $\A_i$ for some $i=0,\dots,T-1$, we define $\A_{i+1}$ to be a random subset of $\A_i$ obtained by including each element of $\A_i$ into $\A_{i+1}$ with probability $p$ independently. Then both of the following two statements hold.
\begin{itemize}
\item[(a)] For integers $0\le i\le j\le T$, let us consider any outcome of $\A_j$. Then, conditional on this outcome of $\A_j$, for each element $x\in \mathcal{C}$, with probability at least $(j-i)\cdot (1-p)/6$ we have $x\in \A_i$.
\item[(b)] For integers $0\le i< j\le T$ with $j-i+1\le T/2$, let us consider any outcomes of $A_i$ and $\A_j$. Then, conditional on these outcomes of $\A_i$ and $\A_j$, for each element $x\in \A_i$, with probability at least $(T/(j-i))\cdot (1-p)/2$ we have $x\in \A_{j-1}$.
\end{itemize}
\end{lemma}
\begin{proof}
Note that the restrictions $\A_0\cap\{x\}\supseteq \A_1\cap\{x\}\supseteq \dots\supseteq \A_T\cap\{x\}$ of our random sequence are mutually independent between different choices for the element $x\in \mathcal{C}$ (because the random sequence was chosen in a way that decisions of whether to include different elements of $\mathcal{C}$ are completely independent of each other).

For (a), note that for every $x\in \A_j$, when conditioning on $\A_j$, we have $x\in \A_i$ with probability $1\ge T\cdot (1-p)\ge (j-i)\cdot (1-p)/6$. So let us assume that $x\not\in \A_j$. When looking at the event $x\in \A_i$, it is equivalent to condition on any given outcome of $\A_j$ which does not contain $x$ or to just condition on the event $x\not \in \A_j$.  Thus, for (a) it suffices to show that $\Pr[x\in \A_i\mid x\not \in \A_j]\ge (j-i)\cdot (1-p)/6$ for each $x\in \mathcal{C}$.

So let $x\in \mathcal{C}$, then
\[\Pr[x\in \A_i\mid x\not \in \A_j]=\frac{\Pr[x\in \A_i\setminus \A_j]}{\Pr[x\notin \A_j]}\ge \Pr[x\in \A_i\setminus \A_j]=\frac{1}{2}p^i(1-p^{j-i})\ge \frac{p^T}{2}\cdot \frac{1-p^{j-i}}{p^{j-i}}= \frac{p^T}{2}((1/p)^{j-i}-1).\]
Applying Bernoulli's inequality, this implies
\[\Pr[x\in \A_i\mid x\not \in \A_j]\ge \frac{p^T}{2}\cdot (j-i)((1/p)-1)=\frac{p^{T-1}}{2}\cdot (j-i)(1-p)\ge (j-i)(1-p)/6,\]
where in the last step we used that $p^{T-1}\ge (1-\frac{1}{T})^{T-1}\ge 1/e\geq 1/3$. This proves (a).

For (b), we are conditioning on the outcomes of $\A_i$ and $\A_j$, and for some $x\in \A_i$ we are looking at the event $x\in \A_{j-1}$. If $x\in A_j$, then under this conditioning, we have $x\in \A_{j-1}$ with probability $1\ge T \cdot (1-p)\ge (T/(j-i))\cdot (1-p)/2$. So let us assume that $x\in \A_i\setminus \A_j$. It now suffices to prove that $\Pr[x\in \A_{j-1}\mid x \in\A_i\setminus \A_j]\ge (T/(j-i))\cdot (1-p)/2$ (assuming that $j-i+1\le T/2$). Note that
\[\Pr[x\in \A_{j-1}\mid x \in\A_i\setminus \A_j]=\frac{\Pr[x\in \A_{j-1}\setminus \A_j]}{\Pr[x\in \A_i\setminus \A_j]}=\frac{p^{j-1}(1-p)}{p^i (1-p^{j-i})}=(1-p)\cdot \frac{p^{j-i-1}}{1-p^{j-i}}\ge (1-p)\cdot \frac{p^{j-i}}{1-p^{j-i}}.\]
Recalling that $p^{T-1}\ge (1-\frac{1}{T})^{T-1}\ge 1/e$, we have $p^{j-i}\ge \exp(-(j-i)/(T-1))\ge 1-(j-i)/(T-1)$ and hence $1-p^{j-i}\le (j-i)/(T-1)$. Hence, the desired probability $\Pr[x\in \A_{j-1}\mid x \in\A_i\setminus \A_j]$ is at least
\[(1-p)\cdot \frac{p^{j-i}}{1-p^{j-i}}=(1-p) \left(\frac{1}{1-p^{j-i}}-1\right)\ge (1-p) \left(\frac{T-1}{j-i}-1\right)=(1-p)\cdot\frac{T-1-j+i}{j-i}\ge \frac{1-p}{2}\cdot \frac{T}{j-i},\]
where in the last step we used that $T-1-j+i\ge T/2$ as $j-i+1\le T/2$.
\end{proof}

We will also need the following very easy probabilistic lemma.

\begin{lemma}\label{lemma-one-event-holds}
Let $m$ be a positive integer, let $0\le p\le 1/m$ and let $\E_1,\dots,\E_m$ be independent events such that for $i=1,\dots,m$ the event $\E_i$ holds with probability $p_i\ge p$. Then with probability at least $pm/2$,  at least one of the events $\E_1,\dots,\E_m$ holds.
\end{lemma}
\begin{proof}
Since decreasing $p_i$ decreases the relevant probability in the conclusion of the lemma, we can assume that $p_1=\dots=p_m=p$. Now,
\[\Pr[\E_1\vee \dots\vee\E_m]\ge \sum_{i=1}^{m}\Pr[\E_i]-\sum_{i<j}\Pr[\E_i\wedge \E_j]=pm-p^2\binom{m}{2}\ge pm-\frac{p^2m^2}2=pm\left(1-\frac{pm}{2}\right)\ge \frac{pm}{2},\]
as desired.
\end{proof}

We will use the following version of the Chernoff bound for sums of independent indicator random variables, see \cite[Theorem~2.8]{random-graphs} or \cite[Appendix A]{alon-spencer}.

\begin{lemma}\label{lemma-chernoff}
Let $X_1,\dots,X_n$ be independent $\{0,1\}$-valued random variables. Let $X=X_1+\dots+X_n$, and $\mu=\mathbb{E}[X]$. Then for every $\delta\ge 0$, we have
\[\Pr[X\le (1-\delta)\mu]\le e^{-\delta^2\mu/2}\quad \text{and}\quad \Pr[X\ge (1+\delta)\mu]\le e^{-\delta^2\mu/(2+\delta)}.\]
\end{lemma}

In particular, the following corollary will be convenient for us.

\begin{corollary}\label{corollary-chernoff}
Let $X$ and $\mu$ be as in Lemma \ref{lemma-chernoff}, and let $\tau\ge 2\mu$. Then we have
\[ \Pr[X\ge \tau]\le e^{-\tau/6}.\]
\end{corollary}
\begin{proof}
If $\mu=0$, then we have $X=0$ with probability $1$, and so the statement is trivially true. So let us assume that $\mu>0$. Let $\delta=(\tau/\mu) -1\ge 1$, and note that then $2+\delta=1+(\tau/\mu)=(\tau+\mu)/\mu$. Now, by the second inequality in Lemma \ref{lemma-chernoff} we have
\begin{align*}
\Pr[X\ge \tau]=\Pr[X\ge (1+\delta)\mu]&\le \exp\left(-\frac{\delta^2}{2+\delta}\cdot \mu\right)=\exp\left(-\frac{(\tau-\mu)^2/\mu^2}{(\tau+\mu)/\mu}\cdot \mu\right)=\exp\left(-\frac{(\tau-\mu)^2}{\tau+\mu}\right)\\
&\le \exp\left(-\frac{(\tau-(\tau/2))^2}{\tau+(\tau/2)}\right)= \exp\left(-\frac{\tau^2/4}{(3/2)\cdot \tau}\right)=\exp(-\tau/6),
\end{align*}
where for the second inequality we used that $\mu\le \tau/2$.
\end{proof}

\subsection{A probabilistic lemma for bipartite graphs}

Our last preparatory lemma is a probabilistic lemma concerning certain events assigned to the edges of some bipartite graph.

\begin{lemma}\label{lemma-prob-graph}
Let $0\le \eps\le 1$ and let $1\le \ell\le d$ be integers, and consider a bipartite graph $G$ with vertex set $U$ on one side and vertex set $W$ on the other side. Let us assume that $G$ has at least $\eps d |U|$ edges, but that every vertex $w\in W$ satisfies $\deg(w)\le d$. For every edge $e\in E(G)$, consider an event $\mathcal{E}_e$ which holds with probability at least $\ell/d$. Suppose that for every vertex $v\in U\cup W$ the events $\mathcal{E}_e$ are mutually independent for all edges $e$ incident with $v$. Then with probability at least $1-15\exp(-\eps\ell/16)$, there exists a subset $E'\su E(G)$ of size $|E'|\ge (\eps/12)\cdot \ell\cdot |U|$ such that $\mathcal{E}_e$ holds for each $e\in E'$ and such that for every vertex $w\in W$ we have $\deg_{E'}(w)\le 2\ell$.
\end{lemma}

\begin{proof}
We may assume that we have $\Pr[\mathcal{E}_e]=\ell/d$ for all edges $e\in E(G)$. Indeed, suppose we already proved the lemma in this special case. Then, in the general case that $\Pr[\mathcal{E}_e]\ge \ell/d$ for every edge $e\in E(G)$, let us consider an event $\mathcal{E}'_e$ with $\Pr[\mathcal{E}'_e]=(\ell/d)/\Pr[\mathcal{E}_e]$ for each $e\in E(G)$ such that these events $\mathcal{E}'_e$ are mutually independent and also independent from the events $\mathcal{E}_e$. Note that then for all  $e\in E(G)$ we have $\Pr[\mathcal{E}_e \land \mathcal{E}'_e]=\Pr[\mathcal{E}_e]\cdot \Pr[\mathcal{E}'_e]=\ell/d$. Furthermore, for every vertex $v\in U\cup W$, the events $\mathcal{E}_e \land \mathcal{E}'_e$ are mutually independent for all edges $e$ incident with $v$. Applying the special case of the lemma to the events $\mathcal{E}_e \land \mathcal{E}'_e$ (with $\Pr[\mathcal{E}_e \land \mathcal{E}'_e]=\ell/d$ for all $e\in E(G)$) we find that with probability at least $1-15\exp(-\eps\ell/16)$, there exists a subset $E'\su E(G)$ of size $|E'|\ge (\eps/12)\cdot \ell\cdot |U|$ such that $\mathcal{E}_e \land \mathcal{E}'_e$ holds for each $e\in E'$ and such that for every vertex $w\in W$ we have $\deg_{E'}(w)\le 2\ell$. Any such set $E'$ now has the desired conditions, since having $\mathcal{E}_e \land \mathcal{E}'_e$ for each $e\in E'$ in particular implies having $\mathcal{E}_e$ for each $e\in E'$. Thus, it indeed suffices to prove the lemma in the special case that $\Pr[\mathcal{E}_e]=\ell/d$ for all $e\in E(G)$.

So let us now assume that $\Pr[\mathcal{E}_e]=\ell/d$ for all $e\in E(G)$. Let $U'\su U$ be the set of vertices $u\in U$ with $\deg(u)\ge \eps d/2$. Then
\[\sum_{u\in U\setminus U'} \deg(u)< |U\setminus U'|\cdot \frac{\eps d}{2} \le \frac{\eps d |U|}{2}\le \frac{|E(G)|}{2},\]
and consequently
\[\sum_{u\in U'} \deg(u)\ge \sum_{u\in U} \deg(u)-\frac{|E(G)|}{2}=|E(G)| -\frac{|E(G)|}{2}=\frac{|E(G)|}{2}.\]

Let us call a vertex $u\in U'$ \emph{bad} if there are fewer than $\ell \deg(u)/(2d)$ edges $e\in E(G)$ incident with $u$ for which the event $\E_e$ holds (note that this depends on the outcomes of these random events). Recall that the events  $\E_e$ are independent for all edges $e$ incident with $u$. The expected number of edges $e\in E(G)$ incident with $u$ with $\E_e$ holding is precisely $(\ell/d)\cdot \deg(u)=\ell \deg(u)/d$. Hence, by the Chernoff bound (see Lemma \ref{lemma-chernoff}), for each vertex $u\in U'$ the probability that $u$ is bad is at most $\exp(-\ell\deg(u)/(8d))\le \exp(-\eps \ell/16)$ (recall that $\deg(u)\ge \eps d/2$ since $u\in U'$).

Thus, we obtain
\[\mathbb{E}\left[\sum_{u\in U'\text{ bad}} \deg(u)\right]\le \exp(-\eps \ell/16)\cdot \sum_{u\in U'} \deg(u).\]
So by Markov's inequality, we have
\[\Pr\left[\sum_{u\in U'\text{ bad}} \deg(u)\ge \frac{1}{3}\cdot \sum_{u\in U'} \deg(u)\right]\le 3 \exp(-\eps \ell/16),\]
and hence
\[ \Pr\left[\sum_{u\in U'\text{ not bad}} \deg(u)\ge \frac{2}{3}\cdot \sum_{u\in U'} \deg(u)\right]\ge 1-3 \exp(-\eps \ell/16).\]
Recalling $\sum_{u\in U'} \deg(u)\ge |E(G)|/2$, this implies
\begin{equation}\label{eq-sum-not-bad}
\Pr\left[\sum_{u\in U'\text{ not bad}} \deg(u)\ge \frac{|E(G)|}{3}\right]\ge 1-3 \exp(-\eps \ell/16).
\end{equation}

Let us call an edge $e\in E(G)$ of the form $(u,w)$ with $u\in U$ and $w\in W$ \emph{crowded} if the event $\E_e$ holds and there are at least $2\ell$ other edges $e'\in E(G)\setminus \{e\}$ incident with $w$ for which $\E_{e'}$ holds. We claim that every edge $e\in E(G)$ (of the form $(u,w)$ with $u\in U$ and $w\in W$) is crowded with probability at most $ \exp(-\ell/3)\cdot (\ell/d)$. Indeed, the expected number of edges $e'\in E(G)\setminus \{e\}$ incident with $w$ for which $\E_{e'}$ holds is $(\deg(w)-1)\cdot (\ell/d)<\ell$ (recalling that $\deg(w)\le d$). Furthermore, the events $\E_{e'}$ are independent for all $e'\in E(G)\setminus \{e\}$ incident with $w$ and are also independent from $\E_{e}$. So by the Chernoff bound (see Corollary \ref{corollary-chernoff}) the probability that there are at least $2\ell$ edges $e'\in E(G)\setminus \{e\}$ incident with $w$ for which $\E_{e'}$ holds is at most $\exp(-2\ell/6)=\exp(-\ell/3)$. This is independent of the event $\E_e$, which has probability $\ell/d$. Hence, the probability that both of these conditions hold, i.e.\ that $e$ is crowded, is at most $\exp(-\ell/3)\cdot (\ell/d)$.

Since every edge $e\in E(G)$ is crowded with probability at most $\exp(-\ell/3)\cdot (\ell/d)$, the expected number of crowded edges $e\in E(G)$ is at most $\exp(-\ell/3)\cdot (\ell/d)\cdot |E(G)|$. Hence, by Markov's inequality, the probability that there are at least $(\ell/d)\cdot (|E(G)|/12)$ crowded edges $e\in E(G)$ is at most $12 \exp(-\ell/3)\le 12 \exp(-\eps\ell/16)$.

Together with (\ref{eq-sum-not-bad}) this implies that with probability at least $1-15 \exp(-\eps \ell/16)$ there are at most $(\ell/d)\cdot (|E(G)|/12)$ crowded edges $e\in E(G)$ and we have
\[\sum_{u\in U'\text{ not bad}} \deg(u)\ge \frac{|E(G)|}{3}.\]
It now suffices to show that whenever this happens there is a subset $E'\su E$ satisfying the conditions in the statement of the lemma. Let $E'\su E(G)$ be the set of edges $e\in E(G)$ such that $\E_e$ holds but $e$ is not crowded.

To check that $|E'|\ge (\eps/12)\cdot \ell\cdot |U|$, we first claim that there are at least $(\ell/d)\cdot (|E(G)|/6)$ edges $e\in E(G)$ for which $\E_e$ holds. Indeed, for every vertex $u\in U'$ which is not bad there are at least $\ell\deg(u)/(2d)$ edges $e$ incident with $u$ for which the event $\E_e$ holds. So the total number of edges $e\in E(G)$ for which $\E_e$ holds is at least
\[\sum_{u\in U'\text{ not bad}} \frac{\ell\deg(u)}{2d}=\frac{\ell}{2d}\sum_{u\in U'\text{ not bad}} \deg(u)\ge \frac{\ell}{2d}\cdot \frac{|E(G)|}{3}=\frac{\ell}{d}\cdot \frac{|E(G)|}{6}.\]
Recalling our assumption that there are at most $(\ell/d)\cdot (|E(G)|/12)$ crowded edges $e\in E(G)$, we can conclude that there are at least $(\ell/d)\cdot (|E(G)|/12)$ edges $e\in E(G)$ such that $\E_e$ holds but $e$ is not crowded. Hence, $|E'|\ge (\ell/d)\cdot (|E(G)|/12)\ge (\eps/12)\cdot \ell\cdot |U|$.

By definition of $E'$ we already know that for all edges $e\in E'$ the event $\E_e$ holds. It remains to show that $\deg_{E'}(w)\le 2\ell$ for all $w\in W$. Suppose that there is a vertex $w\in W$ with $\deg_{E'}(w)\ge 2\ell+1$, and let $e\in E'$ be an edge incident with $w$. Then the event $\E_e$ holds and there are at least $2\ell$ other edges in $e'\in E(G)\setminus\{e\}$ incident with $w$ for which the event  $\E_{e'}$ holds. Thus, $e$ is crowded, but this is a contradiction to $e\in E'$. Thus, we indeed have $\deg_{E'}(w)\le 2\ell$ for all $w\in W$.
\end{proof}

\section{Finding a rainbow cycle}

In this section, we prove Theorem \ref{thm-main}, but the key part of the proof will actually be to prove various lemmas (whose proofs are postponed to the next two sections). In order to prove the theorem, it suffices to prove the following proposition.

\begin{proposition}\label{propostion-bound-for-expander}
For sufficiently large $n$, assume that $G$ is a  graph on $n$ vertices with average degree $d(G)\ge 10^7\cdot  \log n\cdot \log \log n$ which is a robust sublinear expander (in the sense of Definition \ref{definition-robust-sublinear-expander}). Then any proper edge-colouring $\gamma:E(G)\to \mathcal{C}$ of the graph $G$ has a rainbow cycle.
\end{proposition}

Indeed, this proposition implies Theorem \ref{thm-main} via the following easy reduction.

\begin{proof}[Proof of Theorem \ref{thm-main} assuming Proposition \ref{propostion-bound-for-expander}]
We may assume that $n$ is sufficiently large (note that this can be ensured by choosing $C$ sufficiently large). So let us assume that $n$ is large enough such that the statement in Proposition \ref{propostion-bound-for-expander} holds for graphs on at least $\log \log n$ vertices. We show that then the statement in Theorem \ref{thm-main} holds with $C=10^8$, so let $G$ be a properly edge-coloured graph on $n$ vertices with average degree $d(G)\ge 10^8\cdot \log n\cdot \log \log n$. Applying Lemma \ref{lemma-find-robust-expander-subgraph}, we obtain a subgraph $H$ of $G$ which is a robust sublinear expander and has average degree
\[d(H)\ge \frac{1}{3}\cdot \frac{\log |V(H)|}{\log n}\cdot d(G)\ge  \frac{1}{3}\cdot \frac{\log |V(H)|}{\log n}\cdot 10^8\cdot \log n\cdot \log \log n>10^7\cdot \log |V(H)|\cdot \log\log n.\]
In particular, $d(H)\ge 10^7\cdot \log |V(H)|\cdot \log\log |V(H)|$. So by Proposition \ref{propostion-bound-for-expander} (noting that $|V(H)|\ge d(H)\ge \log \log n$), the subgraph $H$ must contain a rainbow cycle, and then the original graph $G$ also contains a rainbow cycle.
\end{proof}

The rest of this section is devoted to proving Proposition \ref{propostion-bound-for-expander}, but we will postpone the proof of a key lemma to the next section.

\begin{proof}[Proof of Proposition \ref{propostion-bound-for-expander}]
Let $G$ and $\gamma: E(G)\to \mathcal{C}$ be as in the proposition. Let us fix an arbitrary vertex $x\in V(G)$ for the rest of this proof. A \emph{rainbow walk} in $G$ is a walk on which all edges have distinct colours.

\begin{definition}
For a subset $\A\su \mathcal{C}$ of the colour palette, let $\U(\A)\su V(G)$ be the set of vertices that can be reached from $x$ via a rainbow walk with colours in $\A$.
\end{definition}

Note that for $\A\su \A'\su\mathcal{C}$, we have $\U(\A)\su \U(\A')$. Also, note that for all $\A\su \mathcal{C}$ we have $x\in \U(\A)\su V(G)$  (since $x$ can always be reached by the trivial walk with no edges), and hence $|\!\U(\A)|\ge 1$.

Let
\[K=\left\lceil \log \log n\right\rceil, \quad \quad L=10^5\cdot \lceil \log n\rceil, \quad \text{and}\quad T=KL,\]
and note that $T\le 4\cdot 10^5\cdot \log n\cdot \log \log n\le d(G)/12$ for $n$ sufficiently large. Also, note that $L$ is divisible by $4$.

Let $\A_0\su \mathcal{C}$ be a random subset of the colour palette, obtained by including each colour in $\mathcal{C}$ with probability $1/2$. It now suffices to prove that
\begin{equation}\label{eq-large-rainbow-walk-surrounding-suffices}
\Pr\big[|\!\U(\A_0)|\ge n/\sqrt{e}\big]\ge 2/3.
\end{equation}
Indeed, applying (\ref{eq-large-rainbow-walk-surrounding-suffices}) to the complement $\mathcal{C}\setminus \A_0$ (which has the same distribution as $\A_0$), we also have
\[\Pr\big[|\!\U(\mathcal{C}\setminus\A_0)|\ge n/\sqrt{e} \big]\ge 2/3.\]
Thus, with probability at least $1/3$, we have $|\!\U(\A_0)|+|\!\U(\mathcal{C}\setminus\A_0)|\ge (2/\sqrt{e})\cdot n>n+1$ (assuming that $n$ is sufficiently large), so there must be a vertex $v\in V(G)\setminus\{x\}$ with $v\in \U(\A_0)\cap \U(\mathcal{C}\setminus\A_0)$. This means that $v$ can be reached from $x$ by a rainbow walk with colours in $\A_0$ and also by a rainbow walk with colours in $\mathcal{C}\setminus\A_0$. Concatenating these walks gives a (non-trivial) rainbow closed walk starting and ending at  $x$. Thus, there must be a rainbow cycle. So it indeed suffices to prove (\ref{eq-large-rainbow-walk-surrounding-suffices}).

Let us extend $\A_0$ to a nested sequence of random colour sets $\A_0\supseteq \A_1\supseteq \dots\supseteq \A_T$, where for $i=0,\dots,T-1$, for a given outcome of $\A_i$, we define $\A_{i+1}$ to be a random subset of $\A_i$ obtained by including each colour in $\A_i$ into $\A_{i+1}$ with probability $1-(1/T)$ independently.

Furthermore, let $s:\{0,\dots,K\}\to \mathbb{R}$ be the function given by
\[s(k)=n\cdot \exp\left(-\frac{1}{2}\cdot 10^k\right)\]
for all $k\in \{0,\dots,K\}$. Note that we have $s(0)=n/\sqrt{e}$ and $s(K)\le n\cdot \exp(-(1/2)\cdot (\log n)^2)<1$.

In particular, we always have $|\!\U(\A_{KL})|\ge 1\ge s(K)$. The following lemma states (roughly speaking) that for given $k=1,\dots, K$, if we have $|\!\U(\A_{kL})|\ge s(k)$, then we likely also have $|\!\U(\A_{(k-1)L})|\ge s(k-1)$. Iterating this will then give us a lower bound for $|\!\U(\A_0)|$ with high probability.

\begin{lemma}\label{lemma-step-by-L}
For every $k\in \{1,\dots,K\}$, we have
\[\Pr\big[|\!\U(\A_{(k-1)L})|< s(k-1)\text{ and }|\!\U(\A_{kL})|\ge s(k)\big]\le \frac{1}{4^{k}}.\]
\end{lemma}

We postpone the proof of this lemma to the next section. Intuitively speaking, the relevance of the indices $kL$ for $k=0,\dots,K$ appearing in Lemma \ref{lemma-step-by-L} is that these are the ``spaced-out'' indices for which we apply Markov's inequality (recall the discussion at the end of the proof overview in Section \ref{sec:overview}). Basically, for each such index, we will apply Markov's inequality to show that if the expectation $\EE[|\!\U(\A_{kL})|]$ is small, then $|\!\U(\A_{kL})|$ is small with high probability .

Assuming Lemma \ref{lemma-step-by-L}, we have (recalling that $|\!\U(\A_{KL})|\ge s(K)$ always holds)
\begin{align*}
\Pr[|\!\U(\A_0)|< s(0)]&=\Pr\big[|\!\U(\A_{0})|< s(0)\text{ and }|\!\U(\A_{KL})|\ge s(K)\big]\\
&\le \sum_{k=1}^{K} \Pr\big[|\!\U(\A_{(k-1)L})|< s(k-1)\text{ and }|\!\U(\A_{kL})|\ge s(k)\big]\le \sum_{k=1}^{K} \frac{1}{4^{k}} \le \frac{1}{3}.
\end{align*}
Hence, $\Pr[|\!\U(\A_0)|\ge n/\sqrt{e}]=\Pr[|\!\U(\A_0)|\ge s(0)]\ge 2/3$, which shows (\ref{eq-large-rainbow-walk-surrounding-suffices}) and thereby finishes the proof of Proposition \ref{propostion-bound-for-expander}.
\end{proof}

\section{Proof of Lemma \ref{lemma-step-by-L}}

In this section, we prove Lemma \ref{lemma-step-by-L}, apart from postponing the proof of another lemma to the next section. As in the statement of Lemma \ref{lemma-step-by-L}, let $k\in \{1,\dots,K\}$. Let $\E$ be the event that $|\!\U(\A_{(k-1)L})|< s(k-1)$, and note that this event $\E$ is determined by the outcome of $\A_{(k-1)L}$. We need to show that
\[\Pr\big[\E\text{ and }|\!\U(\A_{kL})|\ge s(k)\big]\le 1/4^k.\]

This inequality is automatically satisfied if $\Pr[\E]\le 1/4^k$. So let us assume from now on that $\Pr[\E]\ge 1/4^k$.

Note that whenever $\E$ holds, we have $1\le |\!\U(\A_{(k-1)L})|< s(k-1)$, so $\Pr[\E]\ge 1/4^k>0$ in particular implies $s(k-1)> 1$. 

Now, assume that the conditional expectation $\EE\big[|\!\U(\A_{kL})|\,\big\vert\, \E\big]$ of $|\!\U(\A_{kL})|$ conditioned on the event $\E$ satisfies $\EE\big[|\!\U(\A_{kL})|\,\big\vert\, \E\big]<s(k)/4^k$. Then by Markov's inequality (in the probability space obtained by conditioning on $\E$), we have
\[\Pr\big[|\!\U(\A_{kL})| \ge s(k)\,\big\vert\, \E\big]\le \frac{\EE\big[|\!\U(\A_{kL})|\,\big\vert\, \E\big]}{s(k)}< \frac{s(k)/4^k}{s(k)}=\frac{1}{4^k}.\]
Thus,
\[\Pr\big[\E\text{ and }|\!\U(\A_{kL})|\ge s(k)\big]\le \Pr\big[|\!\U(\A_{kL})| \ge s(k)\,\big\vert\, \E\big]\le \frac{1}{4^k},\]
as desired in the lemma statement.

So, in addition to $\Pr[\E]\ge 1/4^k$, we may from now on assume that $\EE\big[|\!\U(\A_{kL})|\,\big\vert\, \E\big]\ge s(k)/4^k$.  We will show that these assumptions lead to a contradiction.

Let
\[\eps=\frac{10^{k-1}}{2\log n},\]
and observe that then $s(k-1)=n^{1-\eps}$. We have $0<\eps<1$ (since $s(k-1)> 1$).

The key step to prove Lemma \ref{lemma-step-by-L} is showing the following lemma, whose proof we postpone to the next section. Basically, this lemma states that when conditioning on the event $\mathcal{E}$, for each $j=1,\dots,L/2-1$ we expect $|\!\U(\A_{kL-j-1})|$ to be quite a bit larger than $|\!\U(\A_{kL-j+1})|$, namely larger by a factor of $1+\eps/400$ apart from some error term (also recall the discussion in the proof overview in Section \ref{sec:overview}).

\begin{lemma}\label{lemma-step-by-2-expectation}
Assume that $\Pr[\E]\ge 1/4^k$. Then for each $j=1,\dots,L/2-1$, we have
\[\EE\big[|\!\U(\A_{kL-j-1})|\,\big\vert\, \E\big]\ge \left(1+\frac{\eps}{400}\right)\cdot \EE\big[|\!\U(\A_{kL-j+1})|\,\big\vert\, \E\big] - 15\cdot 4^k\cdot \frac{\eps}{400}\cdot n \cdot \exp\left(-\frac{\eps L}{128}\right).\]
\end{lemma}

This lemma leads to a contradiction as follows. First, note that
\[n \cdot \exp\left(-\frac{\eps L}{128}\right)\le n\cdot \exp\left(-\frac{10^{k-1}}{2\log n}\cdot \frac{10^5\log n}{128}\right)\le n\cdot \exp\left(-\frac{1}{2}\cdot 3\cdot 10^{k}\right)\le n\cdot \exp\left(-\frac{1}{2}\cdot 10^{k}\right)\cdot e^{-10k}\le \frac{s(k)}{60\cdot 16^k}.\]
Recalling our assumption $\EE\big[|\!\U(\A_{kL})|\,\big\vert\, \E\big]\ge s(k)/4^k$, this implies
\[15\cdot 4^k\cdot \frac{\eps}{400}\cdot n \cdot \exp\left(-\frac{\eps L}{128}\right)\le \frac{\eps}{1600}\cdot \frac{s(k)}{4^k}\le \frac{\eps}{1600}\cdot \EE\big[|\!\U(\A_{kL})|\,\big\vert\, \E\big]\le \frac{\eps}{1600}\cdot \EE\big[|\!\U(\A_{kL-j+1})|\,\big\vert\, \E\big]\]
for each $j=1,\dots,L/2-1$. Thus, the inequality in Lemma \ref{lemma-step-by-2-expectation} implies
\[\EE\big[|\!\U(\A_{kL-j-1})|\,\big\vert\, \E\big]\ge \left(1+\frac{3\eps}{1600}\right)\cdot \EE\big[|\!\U(\A_{kL-j+1})|\,\big\vert\, \E\big]\ge \left(1+\frac{\eps}{600}\right)\cdot \EE\big[|\!\U(\A_{kL-j+1})|\,\big\vert\, \E\big]\]
for each $j=1,\dots,L/2-1$. By applying this iteratively to $j=1,3,5,\dots,L/2-1$, we can deduce that
\[\EE\big[|\!\U(\A_{kL-L/2})|\,\big\vert\, \E\big]\ge \left(1+\frac{\eps}{600}\right)^{L/4}\cdot \EE\big[|\!\U(\A_{kL})|\,\big\vert\, \E\big]\ge\left(1+\frac{\eps}{600}\right)^{L/4}\cdot\frac{s(k)}{4^k}.\]
Note that for all real numbers $0\le t\le 1$ we have $1+t\ge \exp(t/2)$. Hence, if $n$ (and hence $K$) is sufficiently large, we have
\[\left(1+\frac{\eps}{600}\right)^{L/4}\ge \exp\left(\frac{\eps}{1200}\cdot \frac{L}{4}\right)\ge \exp\left(\frac{10^{k-1}}{2 \log n}\cdot \frac{10^5\log n}{4\cdot 1200}\right)\ge \exp(10^{k})\ge \exp\left(\frac{1}{2}\cdot 10^{k}\right)\cdot 2\cdot 4^k.\]
So we obtain
\[\EE\big[|\!\U(\A_{kL-L/2})|\,\big\vert\, \E\big]\ge\left(1+\frac{\eps}{600}\right)^{L/4}\cdot\frac{s(k)}{4^k}\ge \exp\left(\frac{1}{2}\cdot 10^{k}\right)\cdot 2\cdot 4^k\cdot \frac{s(k)}{4^k}=2\cdot \exp\left(\frac{1}{2}\cdot 10^{k}\right)\cdot s(k)=2n.\]
But since always $|\!\U(\A_{kL-L/2})|\le |V(G)|\le n$, we must clearly have $\EE\big[|\!\U(\A_{kL-L/2})|\,\big\vert\, \E\big]\le n$, so we indeed obtain a contradiction. This contradiction finishes the proof of Lemma \ref{lemma-step-by-L}, apart from proving Lemma \ref{lemma-step-by-2-expectation}.

\section{Proof of Lemma \ref{lemma-step-by-2-expectation}}

In this section, we prove Lemma \ref{lemma-step-by-2-expectation} which is the last missing piece for the proof of Proposition \ref{propostion-bound-for-expander} (and hence for the proof of our main theorem). 

Let $j\in \{1,\dots,L/2-1\}$. Recall that we chose $k\in \{1,\dots,K\}$ at the start of the previous section, and also recall the definitions of the event $\E$ and the number $0<\eps<1$ in the previous section (and recall that $s(k-1)=n^{1-\eps}$). As in the statement of Lemma \ref{lemma-step-by-2-expectation}, let us assume that $\Pr[\E]\ge 1/4^k$.

Let us consider the event $\mathcal{F}$, determined by $\A_{(k-1)L}$ and $\A_{kL-j}$, defined as follows. Let us say that the event $\mathcal{F}$ holds if both of the following two conditions are satisfied:
\begin{itemize}
\item[(i)] There exists a set $F\su E(G)$ with $|F|\ge (\eps/7)\cdot d(G)\cdot |\!\U(\A_{kL-j})|$ and $\deg_F(v')\le \lceil d(G)\rceil$ for all $v'\in V(G)\setminus \U(\A_{kL-j})$, such that each edge in $F$ is of the form $(v,v')$ with $v\in \U(\A_{kL-j}\!\setminus\! \{\gamma(v,v')\})\su \U(\A_{kL-j})$ and $v'\not\in \U(\A_{kL-j})$.
\item[(ii)] There does not exist a set $F'\su E(G)$ with $|F'|\ge (\eps/100)\cdot L\cdot |\!\U(\A_{kL-j})|$ and $\deg_{F'}(v')\le 2L$ for all $v'\in V(G)\setminus \U(\A_{kL-j})$, such that each edge in $F'$ is of the form $(v,v')$ with $v\in \U(\A_{kL-j}\!\setminus\! \{\gamma(v,v')\})\su \U(\A_{kL-j})$ and $v'\not\in \U(\A_{kL-j})$ and satisfies $\gamma(v,v')\in \A_{(k-1)L}$.
\end{itemize}

Let $\overline{\mathcal{F}}$ denote the complementary event to $\mathcal{F}$, i.e.\ the event that $\mathcal{F}$ does not hold. We will deduce the inequality in Lemma \ref{lemma-step-by-2-expectation} from the following three claims. Intuitively speaking, the first claim says that the event $\mathcal{F}$ is very unlikely, the second claim says that if $\mathcal{F}$ does not hold (which is normally the case) then ``in expectation'' we have $|\!\U(\A_{kL-j-1})|\ge \left(1+\eps/400\right)\cdot |\!\U(\A_{kL-j+1})|$ and the third claim specifies an error term for this inequality in the unlikely case that $\mathcal{F}$ holds.

\begin{claim}\label{claim-1}
We have
\[\Pr[\mathcal{F}]\le 15\exp\left(-\frac{\eps L}{128}\right).\]
\end{claim}

\begin{claim}\label{claim-2}
We have $\Pr[\E\text{ and }\overline{\mathcal{F}}]>0$ and
\[\EE\Big[|\!\U(\A_{kL-j-1})|-\left(1+\frac{\eps}{400}\right)\cdot |\!\U(\A_{kL-j+1})|\,\,\Big\vert\, \,\E,\overline{\mathcal{F}}\Big]\ge 0.\]
\end{claim}

\begin{claim}\label{claim-3}
If $\Pr[\E\text{ and }\mathcal{F}]>0$, then 
\[\EE\Big[|\!\U(\A_{kL-j-1})|-\left(1+\frac{\eps}{400}\right)\cdot |\!\U(\A_{kL-j+1})|\,\,\Big\vert\, \,\E,\mathcal{F}\Big]\ge -\frac{\eps}{400}\cdot n.\]
\end{claim}

Let us now show that these three claims indeed imply the inequality in Lemma \ref{lemma-step-by-2-expectation}. Recalling our assumption that $\Pr[\E]\ge 1/4^k$, in the case that $\Pr[\E\text{ and }\mathcal{F}]>0$ we obtain
\begin{align*}
&\EE\Big[|\!\U(\A_{kL-j-1})|-\left(1+\frac{\eps}{400}\right)\cdot |\!\U(\A_{kL-j+1})|\,\,\Big\vert\, \,\E\Big]\\
&\quad\quad= \EE\Big[|\!\U(\A_{kL-j-1})|-\left(1+\frac{\eps}{400}\right)\cdot |\!\U(\A_{kL-j+1})|\,\,\Big\vert\, \,\E,\mathcal{F}\Big]\cdot \Pr[\mathcal{F}\mid \mathcal{E}]\\
&\quad\quad\quad\quad\quad\quad\quad\quad\quad\quad+ \EE\Big[|\!\U(\A_{kL-j-1})|-\left(1+\frac{\eps}{400}\right)\cdot |\!\U(\A_{kL-j+1})|\,\,\Big\vert\, \,\E,\overline{\mathcal{F}}\Big]\cdot \Pr[\overline{\mathcal{F}}\mid \mathcal{E}]\\
&\quad\quad\ge -\frac{\eps}{400}\cdot n \cdot \Pr[\mathcal{F}\mid \mathcal{E}]+0 \cdot \Pr[\overline{\mathcal{F}}\mid \mathcal{E}]\\
&\quad\quad\ge -\frac{\eps}{400}\cdot n \cdot \frac{\Pr[\mathcal{F}]}{\Pr[\mathcal{E}]}\ge - 15\cdot 4^{k}\cdot \frac{\eps}{400}\cdot n \cdot \exp\left(-\frac{\eps L}{128}\right).
\end{align*}
In the case $\Pr[\E\text{ and }\mathcal{F}]=0$, we simply have
\[\EE\Big[|\!\U(\A_{kL-j-1})|-\left(1+\frac{\eps}{400}\right)\cdot |\!\U(\A_{kL-j+1})|\,\,\Big\vert\, \,\E\Big]=\EE\Big[|\!\U(\A_{kL-j-1})|-\left(1+\frac{\eps}{400}\right)\cdot |\!\U(\A_{kL-j+1})|\,\,\Big\vert\, \,\E,\overline{\mathcal{F}}\Big]\ge 0.\]
So in either case we obtain
\[\EE\Big[|\!\U(\A_{kL-j-1})|-\left(1+\frac{\eps}{400}\right)\cdot |\!\U(\A_{kL-j+1})|\,\,\Big\vert\, \,\E\Big]\ge - 15\cdot 4^{k}\cdot \frac{\eps}{400}\cdot n \cdot \exp\left(-\frac{\eps L}{128}\right),\]
which by linearity of expectation and rearranging yields the inequality in the statement of Lemma \ref{lemma-step-by-2-expectation}.

It remains to prove Claims \ref{claim-1} to \ref{claim-3}. We start by proving Claim \ref{claim-3}, which is the easiest.

\begin{proof}[Proof of Claim \ref{claim-3}]Note that we always have $\U(\A_{kL-j-1})\supseteq \U(\A_{kL-j+1})$ and therefore
\[|\!\U(\A_{kL-j-1})|-\left(1+\frac{\eps}{400}\right)\cdot |\!\U(\A_{kL-j+1})|\ge -\frac{\eps}{400}\cdot |\!\U(\A_{kL-j+1})|\ge -\frac{\eps}{400}\cdot n.\]
This clearly implies the inequality for the conditional expectation in the claim (the assumption $\Pr[\E\text{ and }\mathcal{F}]>0$ ensures that this conditional expectation is well-defined).
\end{proof}

Next, let us prove Claim \ref{claim-1}.

\begin{proof}[Proof of Claim \ref{claim-1}]
It suffices to prove the conditional probability bound $\Pr[\mathcal{F}\mid \A_{kL-j}]\le 15\exp(-\eps L/128)$ for every possible outcome of $\A_{kL-j}$. So let us fix any outcome of $\A_{kL-j}$.

Note that condition (i) in the definition of the event $\mathcal{F}$ only depends on $\A_{kL-j}$ (and not on $\A_{(k-1)L}$). If our fixed outcome $\A_{kL-j}$ does not satisfy (i), then $\Pr[\mathcal{F}\mid \A_{kL-j}]=0$ . Hence, we may assume that the fixed outcome of $\A_{kL-j}$ satisfies (i). Let us fix a subset $F\su E(G)$ as in condition (i). Let $U=\U(\A_{kL-j})$, then $|F|\ge (\eps/7)\cdot d(G)\cdot |U|\ge (\eps/8)\cdot \lceil d(G)\rceil \cdot |U|$  and every edge in $F$ is of the form $(v,v')$ with $v\in U$ and $v'\in V(G)\setminus U$ and $v\in \U(\A_{kL-j}\setminus\{\gamma(v,v')\})$. Furthermore, we have $\deg_F(v')\le  \lceil d(G)\rceil$ for all $v'\in V(G)\setminus U$.

We can now apply Lemma \ref{lemma-prob-graph} to the bipartite graph with edge set $F$ between the vertex sets $U$ and $V(G)\setminus U$, where for each edge $(v,v')\in F$ we consider the event that $\gamma(v,v')\in \A_{(k-1)L}$. These events are independent for edges of different colours, and so in particular they are independent for edges at the same vertex. Furthermore, for every edge $(v,v')\in F$, conditioned on our fixed outcome of $\A_{kL-j}$, by Lemma \ref{lemma-probability-nested-sequence}(a) (applied with $p=1-1/T$) we have $\gamma(v,v')\in \A_{(k-1)L}$ with probability at least $(L-j)/(6T)\ge L/(12T)\ge L/\lceil d(G)\rceil$ (recall that $j\le L/2$ and $d(G)\ge 12T$).

Lemma  \ref{lemma-prob-graph} now shows that with probability at least $1-15\exp(-(\eps/8) L/16)=1-15\exp(-\eps L/128)$ there exists a subset $F'\su F\su E(G)$ of size $|F'|\ge \eps/(8\cdot 12)\cdot L\cdot |U|\ge (\eps/100)\cdot L\cdot |\!\U(\A_{kL-j})|$ such that $\gamma(v,v')\in \A_{(k-1)L}$ for all $(v,v')\in F'$ and $\deg_{F'}(v')\le  2L$ for all $v'\in V(G)\setminus U$. Since $F'\su F$, every edge in $F'$ is of the form $(v,v')$ with $v\in \U(\A_{kL-j}\setminus\{\gamma(v,v')\})\su \U(\A_{kL-j})$ and $v'\in V(G)\setminus \U(\A_{kL-j})$. Hence, $F'$ is as in condition (ii) in the definition of the event $\mathcal{F}$ and so $\mathcal{F}$ does not hold if such a set $F'$ exists. This shows that $\mathcal{F}$ holds with probability at most $15\exp(-\eps L/128)$.
\end{proof}

It remains to prove Claim \ref{claim-2}.

\begin{proof}[Proof of Claim \ref{claim-2}]
For the first part of the claim recall our assumption $\Pr[\E]\ge 1/4^k$ and recall from Claim \ref{claim-1} that
\[\Pr[\mathcal{F}]\le 15 \exp\left(-\frac{\eps L}{128}\right) \le 15 \exp\left(-\frac{10^{k-1}}{2\cdot \log n}\cdot \frac{10^5\log n}{128}\right) \le 15 e^{-10k}<1/4^k.\]
Hence, $\Pr[\E]>\Pr[\mathcal{F}]$ and therefore $\Pr[\E\text{ and }\overline{\mathcal{F}}]>0$.

For the second part of the claim, we need to prove that the expectation of $|\!\U(\A_{kL-j-1})|-\left(1+\eps/400\right)\cdot |\!\U(\A_{kL-j+1})|$ conditional on $\E$ and $\overline{\mathcal{F}}$ is non-negative. It suffices to prove that the expectation of this quantity is non-negative upon further conditioning on any given outcomes of $\A_{(k-1)L}$ and $\A_{kL-j}$ satisfying $\E$ and $\overline{\mathcal{F}}$. In other words, it suffices to prove that
\[\EE\Big[|\!\U(\A_{kL-j-1})|-\left(1+\frac{\eps}{400}\right)\cdot |\!\U(\A_{kL-j+1})|\,\,\Big\vert\, \,\A_{(k-1)L},\A_{kL-j}\Big]\ge 0\]
for every possible outcomes of $\A_{(k-1)L}$ and $\A_{kL-j}$ satisfying $\E$ and $\overline{\mathcal{F}}$. So let us fix $\A_{(k-1)L}\supseteq \A_{kL-j}$ satisfying $\E$ and $\overline{\mathcal{F}}$. Let $U=\U(\A_{kL-j})$.

Since the event $\E$ holds, we have $|U|=|\!\U(\A_{kL-j})|\le |\!\U(\A_{(k-1)L})|< s(k-1)=n^{1-\eps}$. We will now apply Lemma \ref{lemma-auxiliary-graph} with the following auxiliary colouring of the edges of $G$ between $U$ and $V(G)\setminus U$ with the colours red and blue (assume that red and blue are not part of the colour palette $\mathcal{C}$ of our original colouring $\gamma:E(G)\to\mathcal{C}$). For an edge $(v,v')\in E(G)$ with $v\in U=\U(\A_{kL-j})$ and $v'\in V(G)\setminus U$, let us colour $(v,v')$ red in the auxiliary colouring if $v\not\in \U(\A_{kL-j}\setminus \{\gamma(v,v')\})$ and let us colour $(v,v')$ blue in the auxiliary colouring if $v\in \U(\A_{kL-j}\setminus \{\gamma(v,v')\})$. Now, at least one of the two alternatives in Lemma \ref{lemma-auxiliary-graph} must hold, and we split into cases accordingly.

\textbf{Case 1: The first alternative in Lemma \ref{lemma-auxiliary-graph} holds.} In this case, there exists a subset $F_{\text{red}}\su E(G)$ of red edges in $G$ with $|F_{\text{red}}|\ge (\eps/7)\cdot d(G)\cdot |U|$ and $\deg_{F_{\text{red}}}(v)\le \lceil d(G)\rceil$ for all $v\in U$.

After conditioning on $\A_{(k-1)L}$ and $\A_{kL-j}$, for every $(v,v')\in F_{\text{red}}$ the probability of having $\gamma(v,v')\not\in \A_{kL-j+1}$ is at least $1/T\ge 1/\lceil d(G)\rceil$ (since each colour in $\A_{kL-j}$ is missing in $\A_{kL-j+1}$ with probability $1/T$, and each colour that is not in $\A_{kL-j}$ is missing in $\A_{kL-j+1}$ with probability $1$). Furthermore, these events $\gamma(v,v')\not\in \A_{kL-j+1}$ are independent for edges $(v,v')$ of different colours (even after conditioning on $\A_{(k-1)L}$ and $\A_{kL-j}$).

Thus, for each $v\in U$ we can apply Lemma \ref{lemma-one-event-holds} to the events $\gamma(v,v')\not\in \A_{kL-j+1}$ for the $\deg_{F_{\text{red}}}(v)\le \lceil d(G)\rceil$ different edges $(v,v')\in F_{\text{red}}$. By the lemma, with probability at least $\deg_{F_{\text{red}}}(v)/(2\lceil d(G)\rceil)$, we have $\gamma(v,v')\not\in \A_{kL-j+1}$ for at least one edge $(v,v')\in F_{\text{red}}$. Thus, for every $v\in U$, we showed that
\begin{equation}\label{eq-case-1-some-v}
\Pr\big[\gamma(v,v')\not\in \A_{kL-j+1}\text{ for some edge }(v,v')\in F_{\text{red}}\text{ at }v\,\,\big\vert\, \,\A_{(k-1)L},\A_{kL-j}\big]\ge \frac{\deg_{F_{\text{red}}}(v)}{2\lceil d(G)\rceil}.
\end{equation}

We claim that if $(v,v')\in F_{\text{red}}$ (with $v\in U$ and $v'\in V(G)\setminus U$) satisfies $\gamma(v,v')\not\in \A_{kL-j+1}$, then $v\not\in \U(\A_{kL-j+1})$. Indeed, under the assumption $\gamma(v,v')\not\in \A_{kL-j+1}$ we have $\A_{kL-j+1}\su \A_{kL-j}\setminus \{\gamma(v,v')\}$ and hence $\U(\A_{kL-j+1})\su \U(\A_{kL-j}\setminus \{\gamma(v,v')\})$. Since $(v,v')$ is red in the auxiliary colouring, we know that $v\not\in \U(\A_{kL-j}\setminus \{\gamma(v,v')\})$ and can conclude that $v\not\in \U(\A_{kL-j+1})$.

Thus, (\ref{eq-case-1-some-v}) implies
\[\Pr\big[v\not\in \U(\A_{kL-j+1})\,\,\big\vert\, \,\A_{(k-1)L},\A_{kL-j}\big]\ge \frac{\deg_{F_{\text{red}}}(v)}{2\lceil d(G)\rceil}.\]
for every $v\in U$. Summing this up for all $v\in U$ yields
\begin{align*}
\EE\big[|U\setminus \U(\A_{kL-j+1})|\,\,\big\vert\, \,\A_{(k-1)L},\A_{kL-j}\big]&=\sum_{v\in U}\Pr\big[v\not\in \U(\A_{kL-j+1})\,\,\big\vert\, \,\A_{(k-1)L},\A_{kL-j}\big]\\
&\ge \sum_{v\in U}\frac{\deg_{F_{\text{red}}}(v)}{2 \lceil d(G)\rceil} =\frac{|F_{\text{red}}|}{2 \lceil d(G)\rceil}\ge \frac{(\eps/7)\cdot d(G)\cdot |U|}{2 \lceil d(G)\rceil}\ge \frac{\eps}{15}\cdot |U|
\end{align*}
(where in the last step we used that $\lceil d(G)\rceil\le 15/14 \cdot  d(G)$ since $n$ is sufficiently large). Therefore we obtain (using that always $\U(\A_{kL-j+1})\su \U(\A_{kL-j})=U$)
\[\EE\big[|\!\U(\A_{kL-j+1})|\,\,\big\vert\, \,\A_{(k-1)L},\A_{kL-j}\big]=|U|-\EE\big[|U\setminus \U(\A_{kL-j+1})|\,\,\big\vert\, \,\A_{(k-1)L},\A_{kL-j}\big]\le \left(1-\frac{\eps}{15}\right)\cdot |U|.\]
On the other hand, we always have $\U(\A_{kL-j-1})\supseteq \U(\A_{kL-j})=U$ and hence $\EE\big[|\!\U(\A_{kL-j-1})|\,\,\big\vert\, \,\A_{(k-1)L},\A_{kL-j}\big]\ge |U|$. Thus,
\[\EE\Big[|\!\U(\A_{kL-j-1})|-\left(1+\frac{\eps}{400}\right)\cdot |\!\U(\A_{kL-j+1})|\,\,\Big\vert\, \,\A_{(k-1)L},\A_{kL-j}\Big]\ge |U|-\left(1+\frac{\eps}{400}\right)\cdot \left(1-\frac{\eps}{15}\right)\cdot |U|\ge 0,\]
as desired.

\textbf{Case 2: The second alternative in Lemma \ref{lemma-auxiliary-graph} holds.} In this case, there exists a subset $F_{\text{blue}}\su E(G)$ of blue edges in $G$ with $|F_{\text{blue}}|\ge (\eps/7)\cdot d(G)\cdot |U|$ and $\deg_{F_{\text{blue}}}(v')\le \lceil d(G)\rceil$ for all $v'\in V(G)\setminus U$. Note that this set $F_{\text{blue}}$ satisfies condition (i) in the definition of the event $\mathcal{F}$. Indeed, for the last part of condition (i) recall that every edge in $F_{\text{blue}}$ is of the form $(v,v')$ with $v\in U=\U(\A_{kL-j})$ and $v'\not\in \U(\A_{kL-j})$, and also $v\in \U(\A_{kL-j}\setminus \{\gamma(v,v')\})$ since $(v,v')$ is blue.

Now, since we chose outcomes of $\A_{(k-1)L}$ and $\A_{kL-j}$ that do not satisfy $\mathcal{F}$, condition (ii) in the definition of $\mathcal{F}$ must fail. This means that there exists a set $F'\su E(G)$ with $|F'|\ge (\eps/100)\cdot L\cdot |U|$ and $\deg_{F'}(v')\le 2L$ for all $v'\in V(G)\setminus U$, such that each edge in $F'$ is of the form $(v,v')$ with $v\in \U(\A_{kL-j}\!\setminus\! \{\gamma(v,v')\})\su U$ and $v'\not\in U$ and satisfies $\gamma(v,v')\in \A_{(k-1)L}$.

We claim for every $(v,v')\in F'$ the probability (conditioning on $\A_{(k-1)L}$ and $\A_{kL-j}$) of having $\gamma(v,v')\in \A_{kL-j-1}$ is at least $1/(2L)$. Indeed, since $(v,v')\in F'$, we have $\gamma(v,v')\in \A_{(k-1)L}$. So by Lemma \ref{lemma-probability-nested-sequence}(b) (applied with $p=1-1/T$), the probability of having $\gamma(v,v')\in \A_{kL-j-1}$ is at least $T/(L-j)\cdot 1/(2T)\ge 1/(2L)$. Furthermore, these events $\gamma(v,v')\in \A_{kL-j-1}$ are independent for edges $(v,v')$ of different colours (even after conditioning on $\A_{(k-1)L}$ and $\A_{kL-j}$).

Now, for each $v'\in V(G)\setminus U$ we can apply Lemma \ref{lemma-one-event-holds} to the events $\gamma(v,v')\in \A_{kL-j-1}$ for the $\deg_{F'}(v')\le 2L$ different edges $(v,v')\in F'$. By Lemma \ref{lemma-one-event-holds}, with probability at least $\deg_{F'}(v')/(4L)$, we have $\gamma(v,v')\in \A_{kL-j-1}$ for at least one edge $(v,v')\in F'$. Thus, for every $v'\in V(G)\setminus U$ , we showed that
\begin{equation}\label{eq-case-2-some-v-prime}
\Pr\big[\gamma(v,v')\in \A_{kL-j-1}\text{ for some edge }(v,v')\in F'\text{ at }v'\,\,\big\vert\, \,\A_{(k-1)L},\A_{kL-j}\big]\ge \frac{\deg_{F'}(v)}{4L}.
\end{equation}

We claim that if $(v,v')\in F'$ (with $v\in U$ and $v'\in V(G)\setminus U$) satisfies $\gamma(v,v')\in \A_{kL-j-1}$, then $v'\in \U(\A_{kL-j-1})$. Indeed, because $(v,v')\in F'$, by the choice of $F'$ we have $v\in \U(\A_{kL-j}\setminus \{\gamma(v,v')\})$. This means that there exists a rainbow walk from $x$ to $v$ with colours in $\A_{kL-j}\setminus \{\gamma(v,v')\}\su \A_{kL-j-1}\setminus \{\gamma(v,v')\}$. Adding the edge $(v,v')$ gives a rainbow walk from $x$ to $v'$ with colours in $\A_{kL-j-1}$ (here, we used that $\gamma(v,v')\in \A_{kL-j-1}$). Thus, indeed $v'\in \U(\A_{kL-j-1})$.

Thus, (\ref{eq-case-2-some-v-prime}) implies
\[\Pr\big[v'\in \U(\A_{kL-j-1})\,\,\big\vert\, \,\A_{(k-1)L},\A_{kL-j}\big]\ge \frac{\deg_{F'}(v')}{4L}.\]
for every $v'\in V(G)\setminus U$. Summing this up for all $v'\in V(G)\setminus U$ yields
\begin{align*}
\EE\big[|\!\U(\A_{kL-j-1})\setminus U|\,\,\big\vert\, \,\A_{(k-1)L},\A_{kL-j}\big]&=\sum_{v'\in V(G)\setminus U}\Pr\big[v'\in \U(\A_{kL-j-1})\,\,\big\vert\, \,\A_{(k-1)L},\A_{kL-j}\big]\\
&\ge \sum_{v'\in V(G)\setminus U}\frac{\deg_{F'}(v')}{4L} =\frac{|F'|}{4L}\ge \frac{(\eps/100)\cdot L\cdot |U|}{4L}\ge \frac{\eps}{400}\cdot |U|.
\end{align*}
Therefore we obtain (using that always $\U(\A_{kL-j-1})\supseteq \U(\A_{kL-j})=U$)
\[\EE\big[|\!\U(\A_{kL-j-1})|\,\,\big\vert\, \,\A_{(k-1)L},\A_{kL-j}\big]=|U|+\EE\big[|\!\U(\A_{kL-j-1})\setminus U|\,\,\big\vert\, \,\A_{(k-1)L},\A_{kL-j}\big]\ge \left(1+\frac{\eps}{400}\right)\cdot |U|.\]
On the other hand, we always have $\U(\A_{kL-j+1})\su \U(\A_{kL-j})=U$ and hence $\EE\big[|\!\U(\A_{kL-j+1})|\,\,\big\vert\, \,\A_{(k-1)L},\A_{kL-j}\big]\le |U|$. Thus,
\[\EE\Big[|\!\U(\A_{kL-j-1})|-\left(1+\frac{\eps}{400}\right)\cdot |\!\U(\A_{kL-j+1})|\,\,\Big\vert\, \,\A_{(k-1)L},\A_{kL-j}\Big]\ge \left(1+\frac{\eps}{400}\right)\cdot|U|-\left(1+\frac{\eps}{400}\right)\cdot |U|\ge 0,\]
as desired.

This finishes the proof of Claim \ref{claim-2} in both cases.
\end{proof}

This completes the proof of Lemma \ref{lemma-step-by-2-expectation} and hence of Proposition \ref{propostion-bound-for-expander}.

\section{Almost rainbow cycles}
\label{sec:almost-rainbow}

A cycle in a properly edge-coloured graph is said to be \textit{$r$-almost rainbow} if there are at most $r$ edges of each colour on the cycle. So, for example, a $1$-almost rainbow cycle is the same as a rainbow cycle, and any cycle of length $\ell$ is trivially $\ell$-almost rainbow. This offers a natural intermediate notion between a rainbow cycle and an unrestricted cycle, shedding light on the transition from the simple fact that $n$ edges in an $n$-vertex graph guarantee a cycle to the fact that at least $\Omega(n \log n)$ edges are needed to guarantee a rainbow cycle. Somewhat surprisingly, the behaviour is already slightly different if we allow colours to be used twice rather than once, as shown in Proposition~\ref{prop:almost-rainbow} for $r=2$. As we allow colours to be used more and more times, we observe a gradual progression up to allowing $\Theta(\log n)$ edges of any colour, at which point already a linear number of edges suffices.  This transition is described by Proposition~\ref{prop:almost-rainbow}, which we prove below.

Recall that $\tr_{r}(n,\C)$ denotes the maximal possible number of edges in a properly edge-coloured $n$-vertex graph without an $r$-almost rainbow cycle.

\begin{proof}[Proof of Proposition~\ref{prop:almost-rainbow}] As in the statement of the proposition, let $2\le r \le \frac{1}{10}  \log n$. Let us start with the lower bound for $\tr_{r}(n,\C)$. By a classical result of Erd\H{o}s and Sachs \cite[p.\ 252]{ES}, for any integers $d\ge 2$ and $r\ge 2$ and $n\ge 4d^{(d+1)r}$, there exists an $n$-vertex graph with girth at least $(d+1)r+1$, maximum degree $d$ and at least $\Omega(nd)$ edges (in fact, they proved that for even $n\ge 4d^{(d+1)r}$ there is a $d$-regular graph on $n$ vertices with girth at least $(d+1)r+1$). By Vizings's theorem, this graph has a proper edge-colouring using at most $d+1$ colours. In this colouring, every cycle contains some colour at least $r+1$ times (since the cycle has length more than $(d+1)r$ and there are at most $d+1$ colours in total). Choosing $d:=\left\lfloor{\frac{\log n}{2r}}/{\log \frac{\log n}{r}}\right\rfloor \ge 2$ (recalling that $r\le \frac{1}{10}  \log n$), we obtain (noting that then we indeed have $4d^{(d+1)r}\le 4d^{(3/2)dr}\le 4n^{3/4}\le n$)
\[\tr_{r}(n,\C)\ge \Omega(nd)\ge \Omega\left(n \cdot \frac{(\log n)/r}{\log ((\log n)/r)}\right).\]

Turning to the upper bound, let $d:=16\left\lceil{\frac{\log n}{r}}/{\log \frac{\log n}{r}}\right\rceil$, so that $\frac18 rd\log \frac d2 > \log n$ and hence $(d/2)^{dr/8}>n$. It suffices to show that $\tr_{r}(n,\C)\le nd$, since then we have
\[\tr_{r}(n,\C)\le nd=16n\left\lceil{\frac{\log n}{r}}/{\log \frac{\log n}{r}}\right\rceil\le O\left(n \cdot \frac{(\log n)/r}{\log ((\log n)/r)}\right),\]
recalling that $(\log n)/r\ge 10$. To show $\tr_{r}(n,\C)\le nd$, we need to show that in any properly edge-coloured graph with average degree more than $2d$ there exists a cycle using each colour at most $r$ times. To this end, let us repeatedly remove any vertex of degree at most $d$ until we are left with a (non-empty) graph $G$ of minimum degree at least $d+1$. Now fix some vertex $v_0\in V(G)$. We call a sequence of colours $(c_1,\ldots,c_{\ell})$ \emph{admissible} if no colour appears more than $r/2$ many times in the sequence, $c_i \neq c_{i+1}$ for all $i=1,\dots,\ell-1$, and there exists a walk $v_0v_1\ldots v_{\ell}$ in the graph $G$ (starting at $v_0$) such that the edge $v_{i-1}v_{i}$ has colour $c_{i}$ for all $i=1,\dots,\ell$. We note that such a walk is unique for any admissible sequence and we refer to it as the corresponding walk and to its endpoint $v_{\ell}$ as the corresponding endpoint of the admissible sequence. It now suffices to show that there are more than $n$ different admissible sequences. Indeed, then there must be two different admissible sequences with the same corresponding endpoint. The union of all edges of the corresponding walks of these two sequences contains at most $r$ edges of each colour, so it includes an $r$-almost rainbow cycle unless this union of the two walks forms a tree. However, if this union was a tree, then both of the walks would need to be the unique path in this tree from $v_0$ to the common endpoint since we forbid backtracking by insisting that $c_i \neq c_{i-1}$. Since the two walks cannot be identical, this is a contradiction, so the union of the two walks cannot form a tree and we have an $r$-almost rainbow cycle as desired.

To see that we have more than $n$ different admissible sequences, let us define $L=\frac{d}2\cdot \lfloor \frac r2 \rfloor$, and show by induction that there are at least $(d/2)^{\ell}$ admissible sequences for any given length $\ell =1,\dots,L$. For $\ell=1$, any of the colours appearing on the at least $d+1$ edges incident to $v_0$ gives rise to a length-$1$ admissible sequence, establishing the base case of the induction. Now suppose that we have at least $(d/2)^{\ell-1}$ admissible sequences of length $\ell-1$. For each such sequence $(c_1,\ldots,c_{\ell-1})$, the corresponding endpoint $v_{\ell-1}$ is incident to at least $d$ edges with distinct colours other than $c_{\ell-1}$. At most $d/2$ of these at least $d$ colours can appear $\lfloor r/2\rfloor$ times among $(c_1,\ldots,c_{\ell-1})$. Any other colour can be used to extend our sequence into an admissible sequence of length $\ell$, implying that there are indeed at least $(d-d/2) \cdot (d/2)^{\ell-1}=(d/2)^{\ell}$ admissible sequences of length $\ell$. Finally, observe that by $r\ge 2$ we have $L=\frac{d}2\cdot \lfloor \frac r2 \rfloor \ge dr/8$, so in total we have at least $(d/2)^{L}\ge (d/2)^{dr/8}>n$ different admissible sequences.
\end{proof}

Note that Proposition~\ref{prop:almost-rainbow} shows that for any $r\ge 2$, the $r$-almost rainbow Tur\'{a}n number $\tr_{r}(n,\C)$ is by at least a $\Theta(\log \log n)$ factor smaller than the rainbow Tur\'{a}n number $\tr_{1}(n,\C)$. Indeed, we have $\tr_{1}(n,\C)\ge \Omega(n\log n)$ (which is equivalent to the $\Omega(\log n)$ lower bound for the best possible average degree in the setting of \Cref{thm-main}). This can be seen by considering a hypercube graph with vertex set $\{0,1\}^m$, where two vertices are connected by an edge if and only if they differ in exactly one coordinate. Taking the colour palette $\{1,\dots,m\}$, let us colour any such edge with the colour given by the index of the coordinate in which the two vertices differ (that is, the colour corresponds to the coordinate direction of the edge when visualizing the hypercube geometrically). This gives a proper edge-colouring of the hypercube graph which does not have any rainbow cycles, as every cycle must contain every colour an even number of times.

This example also shows that Proposition~\ref{prop:almost-rainbow} cannot be extended to the notion of almost rainbow cycles considered by Janzer and Sudakov \cite{JS-rainbow}, where most colours appear only once, but a small proportion of edges on a cycle are allowed to use non-unique colours. The hypercube example above shows that there are properly edge-coloured graphs on $n$ vertices with $\Omega(n\log n)$ edges that do not have any such cycles, even if we only insist on having at least one colour appearing exactly once on the cycle.
    
\section{Applications and connections to additive number theory}
\label{sec:additive}
In this section, we continue the discussion of applications and connections to additive number theory started in \Cref{sec:intro-additive}, and in particular show how \Cref{thm:additive-main} follows from \Cref{thm-main}.

\subsection{Upper bounds on the additive dimension}

Recall that given an arbitrary group $(G,\cdot)$ with identity element $e$, we say a subset $S \subseteq G$ is dissociated if there is no solution to $g_1^{\eps_1} \dotsm g_m^{\eps_m}=e$ with distinct $g_1,\ldots, g_m \in S$ as well as $\eps_1,\dots,\eps_m \in \{-1,1\}$ and $m\ge 1$. The additive dimension $\dim A$ of a finite set $A \subseteq G$ is defined as the maximum cardinality of a dissociated subset of $A$. We will refer to a product of the form $g_1^{\eps_1} \dotsm g_m^{\eps_m}=e$ with distinct $g_1,\dots,g_m\in G$ and $\eps_1,\dots,\eps_m \in \{-1,1\}$ and $m\ge 1$, as a \emph{non-trivial product}.

The following observation shows how a dissociated subset of a group can be used to construct a properly edge-coloured graph without rainbow cycles. With the same idea, \Cref{thm-main} can be used to obtain upper bounds for the size of dissociated subsets, and in particular to prove \Cref{thm:additive-main}.

\begin{observation}\label{obs:connection}
Let $G$ be a group of size $|G|=n$, and let $S\su G$ be a dissociated subset of $G$. Then there exists a properly edge-coloured $|S|$-regular graph on $2n$ vertices without a rainbow cycle. Furthermore, if $S$ consists only of elements of even order in the group $G$, then there also exists a properly edge-coloured $|S|$-regular graph on $n$ vertices without a rainbow cycle.
\end{observation}
\begin{proof}
For the first part of the observation, let us consider a bipartite graph on $2n$ vertices, where each part consists of $n$ vertices corresponding to the elements of $G$. For every $s\in S$ and every $g\in G$, let us draw an edge from vertex $g$ on the left side to vertex $gs$ on the right side and colour this edge with a colour corresponding to $s$. This way, we obtain an $|S|$-regular graph with a proper edge-colouring (with $|S|$ colours). We claim that this graph does not contain a rainbow cycle. Indeed, suppose there is a rainbow cycle, and let $g_1,\dots,g_m\in S$ be associated with the colours on the edges of the cycle in the order they appear on the cycle (starting at some vertex $g\in G$ on the left side). Note that $m$ is even (as the graph is bipartite) and the elements $g_1,\dots,g_m\in S$ are distinct (as the cycle is rainbow). Now, the vertices on the cycle, starting at $g$, are labelled by $g, gg_1, gg_1g_2^{-1}, gg_1g_2^{-1}g_3,\dots$, and we must have $g=gg_1g_2^{-1}g_3g_4^{-1}\dotsm g_{m-1}g_{m}^{-1}$. Thus, we obtain $g_1g_2^{-1}g_3g_4^{-1}\dotsm g_{m-1}g_{m}^{-1}=e$, contradicting the assumption that the set $S$ is dissociated.

For the second part of the observation, assume that all elements of $S$ have even order. Let $S^{-1}:=\{g^{-1} \mid g \in S \}$,  and consider the Cayley graph of $G$ with the generator set $S \cup S^{-1}$ with $|G|=n$ vertices. For any edge $hh'$ in this graph, there is a unique element $g \in S$ such that $h'=hg$ or $h=h'g$ (otherwise, if $h'=hg_1$ and $h=h'g_2$ for two distinct $g_1,g_2 \in S$, we would obtain $g_1g_2=e$, contradicting the assumption that $S$ is dissociated). Let us now colour any such edge $hh'$ with a colour corresponding to the element $g\in S$ such that $h'=hg$ or $h=h'g$. For every $g\in S$ of order $2$, the edges with the colour corresponding to $G$ form a perfect matching (i.e.\ every vertex is incident to exactly one edge of this colour). For every $g\in S$ of order larger than $2$, the edges with the colour corresponding to $g$ form a collection of vertex-disjoint cycles of even length (their length is precisely the order of $g$) covering all vertices in the graph. For each such cycle, let us delete every other edge so that in the resulting graph every vertex is incident to exactly one edge of each colour. Hence the resulting graph is $|S|$-regular with a proper edge-colouring. It remains to show that there is no rainbow cycle in this graph. Indeed, assume that the vertices $h_1, h_2, \dots, h_m$ form a rainbow cycle in this order. Then for $i=1,\dots,m$, we can find an element $g_i\in S$ such that $h_{i+1}=h_ig_i$ or $h_{i+1}=h_ig_i^{-1}$  (where $h_{m+1}:=h_1$). The elements $g_1,\dots,g_m\in S$ must be distinct since they correspond to the colours appearing on the rainbow cycle with vertices $h_1, h_2, \dots, h_m$. Thus, we obtain a solution to $h_1g_1^{\eps_1} \dotsm g_m^{\eps_m}=h_1$, i.e.\ to $g_1^{\eps_1} \dotsm g_m^{\eps_m}=e$, with distinct $g_1,\ldots, g_m \in S$ as well as $\eps_1,\dots,\eps_m \in \{-1,1\}$ and $m\ge 3$. This contradicts the assumption that $S$ is dissociated.
 \end{proof}

Combining \Cref{obs:connection} with our main result in \Cref{thm-main} implies that any dissociated subset $S\su G$ of any group $G$ of size $n\ge 2$ satisfies $|S|\le C\cdot \log (2n)\cdot \log\log (2n)$ for the absolute constant $C>0$ in \Cref{thm-main}. In other words, this shows an upper bound $\dim G\le O(\log |G|\cdot \log\log |G|)$ for the additive dimension of any finite group $G$ of size $|G|\ge 3$ (see also the discussion in Section \ref{sec:intro-additive}).

The same idea as for \Cref{obs:connection} can also be used to deduce \Cref{thm:additive-main} from \Cref{thm-main}. It will be convenient for us to have the following asymmetric variant of \Cref{thm:additive-main}, see also \cite[Theorem 1]{additive_definition} for a similar asymmetric variant of \Cref{thm:sanders} in the abelian setting.

\begin{theorem}
\label{thm:additive-main-asym}
    Let $(G,\cdot)$ be a group and consider finite subsets $A,B\subseteq G$. If $2<|A\cdot B|\le K|A|$ for some positive integer $K$, then we have
    \[\dim B \le O(K \log (K|A|) \log \log (K|A|)).\]
\end{theorem}

\begin{proof}
Note that $2<|A\cdot B|\le K|A|$ implies that $A$ and $B$ are non-empty and $\log ((K+1)|A|) \cdot \log \log ((K+1)|A|)\le O(\log (K|A|) \log \log (K|A|))$. Now suppose towards a contradiction that there exists a dissociated subset $S\subseteq B$ of size $|S|\ge  C \cdot K\cdot \log ((K+1)|A|) \cdot \log \log ((K+1)|A|)$, where $C$ is the constant given by \Cref{thm-main}. 

Consider a bipartite graph $H$ with parts $A$ and $A \cdot B$ and with edge set consisting of all pairs $(a, ab)$ with $a \in A$ and $b\in S$. We consider a proper edge-colouring where each edge $(a, ab)$ is given a colour corresponding to $b$. Note that the resulting properly-edge coloured graph $H$ is a subgraph of the bipartite graph considered in the proof of the first part of \Cref{obs:connection}, and in particular $H$ does not contain a rainbow cycle.

The graph $H$ has $|A|+|A \cdot B|\le (K+1)|A|$ vertices and $|A||S|$ edges, so its average degree is
\[\frac{2|A||S|}{|A|+|A \cdot B|}\ge \frac{2|S|}{K+1} \ge C \cdot \log ((K+1)|A|)\cdot \log \log ((K+1)|A|) \ge C \cdot \log |V(H)| \cdot \log \log |V(H)|.\]
But by \Cref{thm-main} this implies there exists a rainbow cycle in $H$, which is a contradiction. Thus, every dissociated dissociated subset $S\subseteq B$ has size $|S|<  C \cdot K\cdot \log ((K+1)|A|) \cdot \log \log ((K+1)|A|)\le O(K\log (K|A|) \log \log (K|A|))$.
\end{proof}

\Cref{thm:additive-main} follows from Theorem~\ref{thm:additive-main-asym} by setting $B=A$ and observing that we may always assume $K\le |A|$ (since $|A \cdot A| \le |A|^2$) and hence $\log (K|A|) \le 2\log |A|$.

\subsection{Lower bounds on the additive dimension}\label{subsec:lower-bounds}
In this subsection, we briefly turn our attention to lower bounds. Given a group $G$ we denote by $G^k$ the $k$-fold direct product of $G$. Furthermore, as usual, $\S_k$ denotes the symmetric group consisting of the permutations of a set of size $k$.

\begin{observation}\label{obs:lower-bound-product}
Suppose $G$ is a group of size $m$ with $\dim G = d.$ Then $\dim G^k \ge \frac d{\log m} \cdot \log |G^k|$ for any integer $k\ge 1$. In particular, $\dim \S_3^k \ge \frac 3{\log 6} \cdot \log |\S_3^k|$ for any integer $k\ge 1$. Furthermore, for any integer $k\ge 1$, there is a dissociated set $Z\su \S_3^k$ of size $|Z|= \frac 3{\log 6} \cdot \log |\S_3^k|$ such that every element of $Z$ has order $2$ in the group $\S_3^k$.
\end{observation}
\begin{proof}
Let $Y\su G$ be a dissociated subset of size $|Y|=\dim G=d$. Let $Z \subseteq G^k$ consist of the $dk$ elements of the form $(e,\ldots, e, g, e, \ldots, e)\in G^k$ with $g \in Y$, where $e$ is the identity element of $G$. We claim that the set $Z\su G^k$ is dissociated. Indeed, any non-trivial product of elements of $Z$ or their inverses would give, in at least one coordinate, a non-trivial product of elements of $Y$ or their inverses. Since $Y\su G$ is dissociated, such a non-trivial product cannot exist, and hence  $Z\su G^k$ is dissociated. Thus, we have $\dim G^k \ge |Z|=dk = d \cdot \frac{\log |G^k|}{\log |G|}=\frac d{\log m} \cdot \log |G^k|$. Also note that the order of any element $(e,\ldots, e, g, e, \ldots, e)\in Z$ in the group $G^k$ is the same as the order of the element $g\in Y$ in the group $G$.

In the case of $G=\S_3$, we can take the set $Y$ to consist of the three transpositions in $\S_3$. Since this set $Y$ is dissociated, we have $|\dim \S_3|\ge 3$ and therefore $\dim \S_3^k \ge \frac 3{\log 6} \cdot \log |\S_3^k|$. Since the transpositions in $Y$ each have order $2$ in $\S_3$, the set $Z\su \S_3^k$ constructed above consists of elements of order $2$ in $\S_3^k$. As discussed above, this set is dissociated and has size $|Z|=\frac 3{\log 6} \cdot \log |\S_3^k|$.
\end{proof}

Given a lower bound for $\dim G$ for some group $G$, via \Cref{obs:connection} one obtains a lower bound for the problem of determining the minimum possible average degree of a properly edge-coloured graph on $n$ vertices that guarantees the existence of a rainbow cycle. In particular, by the last part of \Cref{obs:lower-bound-product} combined with \Cref{obs:connection}, when $n$ is a power of $6$, there is a properly edge-coloured $n$-vertex graph with average degree at least $\frac{3}{\log 6} \cdot \log n\approx 1.674 \cdot \log n$, which does not contain a rainbow cycle. This is the best-known lower bound for this problem (this bound was obtained in \cite{KMSV} with a different description of the same edge-coloured graph, see also the discussion in the introduction).

A simpler example for a properly edge-coloured $n$-vertex graph with average degree $\Omega(\log n)$ without a rainbow cycle is given by taking hypercube graph and colouring the edges according to the coordinate directions (see the end of Section \ref{sec:almost-rainbow} for a more detailed description of this construction). This construction can also be obtained by combining Observations \ref{obs:lower-bound-product} and \ref{obs:connection}, starting with the group $G=\S_2=\Z_2$, whose additive dimension is $1$ (in $\S_2$, the single transposition forms a dissociated set).

Both these examples have in common that they are formed by taking powers of a symmetric group, namely $\S_2$ or $\S_3$, and that the generator sets for the Cayley graphs consist of transpositions. This raises the natural question of studying the additive dimension $\dim \S_k$ of symmetric groups, which might help in understanding whether the $\log \log n$ term in our main result in \Cref{thm-main} is actually needed. In particular, recalling that both of the constructions above are based on transpositions, one may ask about the additive dimension of the set of transpositions of $\S_k$. We discuss this question in the next subsection.

\subsection{Additive dimension of the set of transpositions}\label{subsec:transpositions}

Here, we discuss the question of determining the additive dimension of the set $T_k$ consisting of all transpositions in the symmetric group $\S_k$. We initially considered this question in order to try to understand whether the $\log \log n$ term is necessary in our result in Theorem~\ref{thm-main}, but it is also a natural question in itself. In particular, since transpositions are self-inverse, the question can equivalently be phrased as follows: What is the minimum number $m$ such that among any set of $m$ transpositions in $\S_k$ we can choose distinct transpositions $\tau_1,\dots, \tau_\ell$ for some $\ell\ge 1$ such that $\tau_1\circ\dots\circ \tau_\ell=\textrm{id}$? In other words, how many transpositions do we need in $\S_k$ to always be able to multiply some of them in some order, without repetitions, to obtain the identity?

This minimum number $m$ is precisely the additive dimension $\dim T_k$ of the set $T_k$ of transpositions in $\S_k$. Since there are only $\binom{k}{2}$ transpositions in total, we trivially have $\dim T_k \le |T_k|=\binom{k}{2}$. There is a simple upper bound of $\dim T_k \le O(k^{7/4})$, based on the observation that for a dissociated set $S\su T_k\su \S_k$, the $k$-vertex graph with edges corresponding to the transpositions in $S$ must be $K_{4,4}$-free. Indeed, suppose this graph would contain a copy of $K_{4,4}$, say with vertices $\{1,2,3,4\}$ on one side and vertices $\{5,6,7,8\}$ on the other side. Noting that
\begin{align*}(15)(26)(16)(25)=(12)(56) &\qquad\qquad\qquad (35)(46)(36)(45)=(34)(56)\\
(17)(28)(18)(27)=(12)(78) &\qquad\qquad\qquad (37)(48)(38)(47)=(34)(78),
\end{align*}
and hence
\[(15)(26)(16)(25)(35)(46)(36)(45)(17)(28)(18)(27)(37)(48)(38)(47)=(12)(56)(34)(56)(12)(78)(34)(78)=\text{id},\]
this would mean that a product of distinct elements of $S$ yields the identity, which is a contradiction to $S$ being dissociated.
With the same idea, one can more generally show that for a dissociated set $S\su T_k\su \S_k$, the graph corresponding to $S$ cannot contain a $2$-blow-up of a cycle as a subgraph. Using a result of Janzer \cite[Theorem~1.14]{oliver-rainbow} on the Tur\'an number of this family of blow-ups, one then obtains $\dim T_k \le O(k^{3/2}(\log k)^{7/2})$. Our \Cref{thm:additive-main-asym} implies the following much stronger bound.
\begin{prop}
\label{prop:trans-upr-bound}
There exists a constant $C>0$ such that for any integer $k\ge 3$, among any set of at least $C \cdot k \cdot (\log k)^2 $ transpositions in $\S_k$, we can choose a non-empty sequence of distinct transpositions whose product is the identity. In other words, we have $\dim T_k \le C \cdot k \cdot (\log k)^2$.
\end{prop}
\begin{proof}
We apply \Cref{thm:additive-main-asym} to $A=\S_k$ and $B=T_k$ and $K=1$, noting that that $|A \cdot B| = |A|$, and obtain $\dim T_k \le O(\log |A| \cdot \log \log |A|)= O(\log k! \cdot \log \log k!)\le O(k \log^2 k)$. 
\end{proof}

The following lower bound for $\dim T_k$ shows the above result is not too far from being tight.

\begin{prop}
\label{prop:trans-lwr-bound}
For any integer $k\ge 3$, there exists a set of $\left(\frac12+o(1)\right)\cdot \frac{k \log k}{\log \log k}$ transpositions in $\S_k$ containing no non-empty sequence with product equal to the identity. In other words, we have $\dim T_k \ge \left(\frac12+o(1)\right)\cdot \frac{k \log k}{\log \log k}$.
\end{prop}

\begin{proof}
Given that $\dim T_k \ge \dim T_{k-1}$, it suffices to prove the result for even $k$. As shown by Erd\H{o}s and Sachs \cite[p.~252]{ES}, for $d\ge 2$ and $g\ge 4$ and any even $k\ge 4d^{g-1}$ there is a $d$-regular graph on $k$ vertices with girth at least $g$.
Hence there is a $d$-regular graph $G$ on the vertex set
$\{1,2, \ldots ,k\}$ with girth $g>d$ and $d\ge (1+o(1)) \frac{\log k}{\log \log k}$.  Let $Z$ be the set of transpositions $(i \: j)$ where $ij$ is an edge of $G$, then $|Z|=\frac{1}{2}kd\ge\left(\frac12+o(1)\right)\cdot \frac{k \log k}{\log \log k}$. 
We claim that there is no sequence of distinct transpositions in $Z$ whose product is the identity. Indeed, assuming there is such a sequence, let $M$ be the set of all vertices $i \in \{1,\dots,k\}$ appearing in at least one transposition in the sequence. For each such vertex $i$, when the first transposition containing $i$ is applied, $i$ is swapped with a neighbour in the graph. Since the same transposition (edge) cannot appear again, and the graph $G$ has girth $g$, every $i\in M$ has to move at least $g$ steps in order to return to its original position. Altogether this corresponds to at least $|M|g$ single moves, and each transposition provides two such moves. The total number of transpositions in the sequence is at most the number of edges in the induced subgraph $G[M]$, which is at most $|M|d/2$ (since $G$ is $d$-regular). It follows that $|M|g\le |M|d/2 \cdot 2$ and therefore $g \le d$, contradicting $g > d$. 
\end{proof}

There is a natural correspondence between a set $S$ of transpositions in $\S_k$ and a graph with vertex set $[k]=\{1,\dots,k\}$ where the edges are formed by the pairs of vertices making the transpositions. In the literature, this graph is called the Schreier graph $\sch(\S_k \curvearrowright [k],S)$ corresponding to the natural action of $\S_k$ on the index set $[k]$ with generator set $S$. 
The proof above used that in order for $S$ to contain a (non-empty) sequence whose product is the identity, the corresponding graph needs to contain a subgraph with a cycle double cover, i.e.\ a set of cycles with the property that any edge belongs to precisely two of these cycles (arising from the trajectories of individual elements of the ground set while applying the sequence of transpositions). Such ``decompositions'' have attracted a lot of attention in connection to the famous cycle double-cover conjecture of Szekeres and Seymour, for more details see e.g.\ the survey \cite{double-cover-survey}. However, in our setting, there are further restrictions. In particular, considering the trajectories of the individual elements of the ground set with their natural direction, we obtain a collection of directed cycles. So if we form a directed graph from a sequence of transposition multiplying to the identity, where for every transposition we draw two directed edges in opposite directions between the corresponding pair of vertices, then this directed graph must have a decomposition into directed cycles. This condition has also been considered in a directed strengthening of the cycle double cover conjecture due to Jaeger \cite{double-cover-survey} from 1985. This relationship and various constructions considered in connection to the cycle double cover conjecture and its directed strengthening might be useful in trying to improve the lower bound in \Cref{prop:trans-lwr-bound}. There is also a further restriction imposed in our setting arising from the ordering of our sequence of transpositions whose product is the identity, but this additional condition seems somewhat difficult to use.

We note that the Schreier graph $\sch(\S_k \curvearrowright [k],S)$, which featured in the proof of the lower bound for $\dim T_k$ in \Cref{prop:trans-lwr-bound}, is not the same graph as the Cayley graph $\text{Cay}(\S_k, S)$ featuring in a more direct deduction of the upper bound for $\dim T_k$ in \Cref{prop:trans-upr-bound} from \Cref{thm-main} via Observation~\ref{obs:connection}. Indeed, the former graph has only $k$ vertices, while the latter has $k!$ vertices. It is worth noting, however, that the relationship between these two graphs plays an important role in understanding the interchange random process, see e.g.\ \cite{aldous-normal} for more details. In particular, these two graphs have the same spectral gap, which is the subject of the famous Aldous conjecture, proved by Caputo, Liggett, and Richthammer \cite{aldous-proof}. Given the close connection between the spectral gap and the expansion properties of a graph, as well as the fact that expansion plays a fundamental role in our proof, it is natural to wonder whether one can exploit the relationship between the graphs $\sch(\S_k \curvearrowright [k],S)$ and $\text{Cay}(\S_k, S)$ to close the gap between our results in Propositions~\ref{prop:trans-upr-bound} and \ref{prop:trans-lwr-bound}.

\section{Concluding remarks}\label{sec:conc-remarks}

In this paper we give an essentially tight bound for the maximum possible average degree of a properly edge-coloured graph on $n$ vertices without a rainbow cycle, a problem raised by Keevash, Mubayi, Sudakov and Verstra\"ete \cite{KMSV}. In particular, we show that for some absolute constant $C>0$ any $n$-vertex properly edge-coloured graph with average degree at least $C\log n \log \log n$ must contain a rainbow cycle, which is tight up to the $\log \log n$ term. The immediate open problem is to decide whether the $\log \log n$ term is necessary.

\begin{qn}
Is the $\log \log n$ factor in the bound for the average degree in \Cref{thm-main} necessary in order to guarantee a rainbow cycle?
\end{qn}

In our proof of \Cref{thm-main}, we first pass to a robust sublinear expander subgraph. Our approach for finding a rainbow cycle is then to randomly split the colour palette into two parts and show that from some fixed starting vertex $x$ we are likely to reach most of the other vertices by rainbow walks using only edges of colours in the first part of the palette, and the same holds for the colours in the second part of the palette (see the discussion in \Cref{sec:overview}, and note that in the proof this statement is encapsulated in \eqref{eq-large-rainbow-walk-surrounding-suffices}). This approach does require the $\log\log n$ term, which may be viewed as circumstantial evidence that the $\log\log n$ term might indeed be needed in \Cref{thm-main}. In fact, for any large enough $n$, there exists an $n$-vertex properly edge-coloured robust sublinear expander with average degree $\log n (\log \log n)^{1-o(1)}$ with the property that when sampling the colours independently, each with probability $\frac12$, and only taking the edges with the sampled colours, with high probability the resulting subgraph has no large connected components. Examples of such robust sublinear expander graphs can be formed by taking Hamming graphs with appropriately chosen parameters, the details of the construction can be found in \cite{X-ray} (these graphs are discussed in \cite{X-ray} as bottlenecks to a similar, but much simpler expansion argument, and the paper \cite{X-ray} only verifies a weaker expansion property for these graphs, but essentially the same argument proves these graphs are also robust sublinear expanders according to our definition). This shows that one cannot hope to push our argument to a bound on the order of $\log n$ without substantial modification (in particular, \eqref{eq-large-rainbow-walk-surrounding-suffices} fails when omitting the $\log \log n$ term in the bound for the average degree in \Cref{thm-main}). As some illustration of the issue our argument encounters, note that the vertex expansion properties of our robust sublinear expanders only guarantee that we reach a constant proportion of vertices, starting from an arbitrary initial vertex, after taking $\Theta(\log n \log \log n)$ steps. Therefore our number of steps needs to be on the order of $\log n \log \log n$ (see the choice of $T$ at the start of the proof of Proposition \ref{propostion-bound-for-expander}), but we need this number to be smaller than the average degree.

On the other hand, we note that these caveats are mostly issues with the probabilistic approach in our proof, and do not rule out the possibility an average degree on the order of $\log n$ suffices to guarantee a rainbow cycle. In fact, even the approach of splitting the colour palette into two parts and showing that from some fixed starting vertex $x$, when restricting to either part of the colour palette, we can reach most of the other vertices by rainbow walks, might work for an average degree on the order of $\log n$. The above discussion shows that one cannot just assign the colours to the two parts of the colour palette randomly, but potentially a more careful choice of the parts of the colour palettes could work. As a first step in this direction, the following question asks whether it is possible to split the colour palette into two parts such that both of the resulting subgraphs corresponding to the two parts of the colour palette are connected.

\begin{qn}\label{qn:expander-color-splitting}
Does there exist a constant $C>0$ such that the edges of any $n$-vertex properly edge-coloured robust sublinear expander $G$ with average degree at least $C \log n$ can be decomposed into two spanning connected subgraphs in such a way that every colour appears on only one of them?
\end{qn}

As mentioned above, the perhaps most natural approach of forming two subgraphs by assigning every colour independently at random to the first or to the second subgraph with probability $\frac{1}{2}$ each, fails (since there are examples for $G$, where the resulting subgraphs are disconnected with high probability). We note that properties of robust sublinear expanders under this type of ``subsampling'' have been quite useful in a number of applications (see e.g.\ \cite{erdos-gallai} and the discussion on this topic there). We note that a positive answer to the above question would not quite suffice to remove the $\log \log n$ term in Theorem~\ref{thm-main}. What would be enough is to ensure there is a pair of vertices that are connected by rainbow paths in both subgraphs. This in turn would be implied if we could show that rather than just being connected, the subgraphs are (perhaps somewhat weaker) robust sublinear expanders. 
A positive answer to \Cref{qn:expander-color-splitting} would nevertheless give a tight answer to a certain ``X-Ray reconstruction'' problem in discrete geometry, when combined with a reduction from \cite{X-ray}. It would also provide an alternative solution to a question of Matou\v{s}ek \cite{matousek} (recently solved by the authors of \cite{norms}), which in turn would give a somewhat weaker version of a result concerning the Erd\H{o}s Unit Distance problem for ``typical'' norms obtained recently in \cite{norms}.

In \Cref{sec:additive}, we discussed a number of interesting connections of the problem of determining the maximum possible average degree of a properly edge-coloured graph on $n$ vertices without a rainbow cycle, including a number of open problems. Let us highlight the problem of determining the maximum possible additive dimension, or equivalently the plus-minus weighted Davenport constant, of a group with $n$ elements.

\begin{qn}
Given a positive integer $n$, what is the smallest number $d=d(n)$ so that in any group $G$ of size $|G|=n$, given any subset $S \subseteq G$ of size $|S|\ge d$, we can always find a solution to $g_1^{\eps_1} \cdot\ldots \cdot g_m^{\eps_m}=e$ with distinct $g_1,\ldots,g_m \in S$ as well as $\eps_1,\ldots,\eps_m \in \{-1,1\}$ and $m\ge 1$. 
\end{qn}

\Cref{thm:additive-main} shows an upper bound of $d(n)\le O(\log n \log \log n)$ and it is again unclear whether the $\log \log n$ term is necessary (it is easy to show a lower bound of the form $d(n)\ge \Omega(\log n)$). We find the instance of this question when $G=\S_k$ to be of particular interest since any $k$-element group can be realised as a subgroup of $\S_k$, since the symmetric group $\S_k$ is far from being abelian, and since the best known lower bounds for $d(n)$ come from exploring dissociated subsets of $\S_k$. More precisely, the best-known lower bounds for $d(n)$ are obtained from dissociated subsets consisting of transpositions, which motivates the following question of determining the additive dimension of the set of all transpositions in the symmetric group $\S_k$. We find this problem very natural and appealing, and it has some interesting connections discussed in \Cref{subsec:transpositions}. 

\begin{qn}\label{qn-transpositions}
Given $k \ge 2$, determine the minimum number $t=t(k)$ such that any set of $t$ transpositions in $\S_k$ contains a non-empty subset whose elements can be multiplied in some order to give the identity.
\end{qn}

\Cref{prop:trans-lwr-bound,prop:trans-upr-bound} show that  $ k (\log k)^{1-o(1)}\le t(k)\le O(k (\log k)^2)$. We note that this question might be helpful in deciding whether the $\log \log n$ term is necessary in our main result in \Cref{thm-main}.  Note that the exponent $2$ in the upper bound $t(k)\le O(k (\log k)^2)$ appears because of applying our main result to a graph on $k!$ vertices, and so the $\log \log (k!)$ term leads to a factor of roughly $\log k$ (while the ``main term'' $\log (k!)$ is roughly $k\log k$). For similar reasons, the previous best bounds in the setting of \Cref{thm-main} in \cite{JS-rainbow,KLLT-rainbow} only lead to upper bounds for \Cref{qn-transpositions} that are worse than the trivial bound $t(k)\le \binom{k}{2}$.

\vspace{0.2cm}
\textbf{Acknowledgments.} We would like to thank Xiaoyu He, Luka Mili\'cevi\'c, and Cosmin Pohoata for useful conversations and Noam Lifshitz for drawing our attention to the connection to the interchange process and the Aldous Conjecture.

\providecommand{\MR}[1]{}
\providecommand{\MRhref}[2]{%
\href{http://www.ams.org/mathscinet-getitem?mr=#1}{#2}}


\providecommand{\bysame}{\leavevmode\hbox to3em{\hrulefill}\thinspace}
\providecommand{\MR}{\relax\ifhmode\unskip\space\fi MR }
\providecommand{\MRhref}[2]{%
  \href{http://www.ams.org/mathscinet-getitem?mr=#1}{#2}
}
\providecommand{\href}[2]{#2}

\end{document}